\documentclass[11pt,times,letter]{article}

\usepackage{amsmath, amsfonts,amsthm}
\usepackage{dsfont}
\usepackage{bbm}
\usepackage{xcolor}

\usepackage[english]{babel}

\usepackage[normalem]{ulem}

\usepackage{latexsym,amssymb,graphicx}
\usepackage{verbatim}
\usepackage{mathrsfs}
\usepackage{epsfig}

\usepackage{xcolor}
\usepackage{tikz}
\usepackage{pgfplots}
\usepgfplotslibrary{fillbetween}

\usepackage{url}
\usepackage{bm}
\usepackage{natbib}

\usepackage[colorlinks,linkcolor=blue,citecolor=blue,anchorcolor=blue]{hyperref}

\usepackage{amssymb}

\def\esssup_#1{\underset{#1}{\mathrm{ess\,sup\, }}}
\def\essinf_#1{\underset{#1}{\mathrm{ess\,inf\, }}}
\def\argmax_#1{\underset{#1}{\mathrm{arg\,max\, }}}
\def\argmin_#1{\underset{#1}{\mathrm{arg\,min\, }}}

\newtheorem{theorem}{Theorem}[section]
\newtheorem{definition}{Definition}[section]
\numberwithin{equation}{section}

\newtheorem{proposition}[theorem]{Proposition}
\newtheorem{assumption}[theorem]{Assumption}
\newtheorem{remark}[theorem]{Remark}
\newtheorem{lemma}[theorem]{Lemma}
\newtheorem{corollary}[theorem]{Corollary}
\newtheorem{example}{Example}[section]
\allowdisplaybreaks[4]

\definecolor{Red}{rgb}{1.00, 0.00, 0.00}

\definecolor{DRed}{rgb}{0.5, 0.00, 0.00}

\definecolor{Blue}{rgb}{0.00, 0.00, 1.00}

\definecolor{Green}{rgb}{0.0, 0.4, 0.0}

\definecolor{mycolor1}{rgb}{0.00000,0.44700,0.74100}%
\definecolor{mycolor2}{rgb}{0.85000,0.32500,0.09800}%
\definecolor{mycolor3}{rgb}{0.92900,0.69400,0.12500}%
\definecolor{mycolor4}{rgb}{0.49400,0.18400,0.55600}%
\definecolor{mycolor5}{rgb}{0.46600,0.67400,0.18800}%
\definecolor{mycolor6}{rgb}{0.00000,1.00000,1.00000}%

\pgfplotsset{compat=newest}

\title{Convex ordering for graphon mean-field systems}

\author{Daorong Cui \thanks{Email: dcuiab@connect.ust.hk, Department of Mathematics, The Hong Kong University of Science and Technology, Clearwater Bay, Kowloon, Hong Kong.}
\and
Shuoqing Deng \thanks{Email: masdeng@ust.hk, Department of Mathematics, The Hong Kong University of Science and Technology, Clearwater Bay, Kowloon, Hong Kong. S. Deng is supported by the Hong Kong University of Science and Technology Start-up grant no. R9826 and Hong Kong RGC Early Career Scheme (ECS) under grant no. 26307125. }
\and
Yang Xiang\thanks{Email: maxiang@ust.hk, Department of Mathematics, The Hong Kong University of Science and Technology, Clearwater Bay, Kowloon, Hong Kong.}
}
\date{\vspace{-1cm}}

\begin{document}
\maketitle

\begin{abstract}

We establish marginal and functional convex order comparisons for graphon mean-field systems on $\mathbb{R}^d$, including the infinite-horizon setting. A key difficulty is that convex order cannot in general be transferred through the usual particle approximation, which requires us to work directly with Euler schemes for the graphon mean-field system. Under a suitable dissipativity condition, we establish Euler approximation estimates that are uniform in both time and the agent label, and combine them with forward-backward induction arguments to obtain the convex order comparisons. For the infinite-horizon problem, dissipativity provides the required uniform-in-time stability, while an exponentially weighted $L^2$-space enables trajectory-level convergence and leads to the functional convex order on $[0,\infty)$. As an application, we derive value-function comparisons for a class of one-dimensional linear-quadratic graphon mean-field games.

\vspace{0.6 cm}


\end{abstract}

\section{Introduction}\label{sec:intro}
Convex order between two distributions $\mu, \nu \in \mathcal{P}_1(\mathbb{R}^d)$, i.e. probability measures on $\mathbb{R}^d$ with finite first moment, is defined by
\begin{equation*}
    \mu \preceq_{cv} \nu \quad \text{if $\forall \varphi: \mathbb{R}^d \rightarrow \mathbb{R}$ convex, } \int_{\mathbb{R}^d} \varphi(x) \mu(\mathrm{d}x) \leq \int_{\mathbb{R}^d} \varphi(y) \nu(\mathrm{d}y).
\end{equation*}
By taking two specific linear functions $\varphi(x) = \pm x$, this naturally implies that both measures have the same mean. Two $\mathbb{R}^d$-valued random vectors $X$ and $Y$ are in convex order if their corresponding distributions are.

The study of convex ordering for stochastic processes dates back to the work of Hajek \cite{Hajek1985} on the case of increasing convex order, and has been explored by various authors since then, see, for instance \cite{Karoui1998, Martini1999, Schied2007,Bergenthum2007} and references therein. We also refer to \cite{LiuPages2023, Pages2014,Jourdain2025,Jourdain2023} for more recent articles. In \cite{Pages2014}, Pag\`{e}s established functional convex order results for two Brownian martingale diffusion processes under the assumption that diffusion coefficient is convex in the spatial variable. Such a convexity condition was further relaxed by  \cite{Jourdain2023} in the one-marginal case while an assumption comparable to the spatial convexity of the diffusion coefficient is needed for the two-marginal case. Liu and Pag\`{e}s \cite{LiuPages2023} extended such functional convex order results to the McKean-Vlasov equations, and Jourdain and Pag\`{e}s \cite{Jourdain2025} compared the convex order of stochastic Volterra equations.

In this paper, we will investigate in depth the comparison of marginal and functional convex orders for scaled graphon mean-field systems, where ``scaled'' means that the drift coefficient is affine in the space variable. Such a system is defined on some filtered probability space $(\Omega,\mathcal{F}, \mathbb{F}, \mathbb{P})$ by
\begin{equation*}
    \mathrm{d}X^u_t = b(t, X^u_t) \mathrm{d} t  + \sigma (t, X^u_t, [G \boldsymbol{\mu}_t]^u) \mathrm{d}B^u_t, \quad u \in I:=[0,1] \text{ and } t \geq 0,
\end{equation*}
where $b$ and $\sigma$ are suitable coefficient functions; $G(\cdot, \cdot)$ is the graphon function, defined as a symmetric measurable function from $I^2$ to $I$. The graphon-weighted measure $[G \boldsymbol{\mu}_t]^u$, defined as
$$
[G \boldsymbol{\mu}_t]^u(\mathrm{d}x) :=  \int_I G(u,v) \mu^v_t (\mathrm{d}x) \mathrm{d}v , \quad \text{$\mu^u_t := \mathcal{L}(X^u_t)$ and $\boldsymbol{\mu}_t := \{\mu^u_t\}_u$},
$$
represents the population's influence on player $u$'s state, and $\{B^u\}_u$ is a sequence of i.i.d. $q$-dimensional standard Brownian motions. As in the McKean-Vlasov dynamics, the coefficients of this system depend on the distributions of the state variables. However, we also remark that each $X^u$ alone is not a standard McKean-Vlasov dynamic, as its own law $\mu^u$ plays a negligible role in the evolution, see also \cite{BayrChak2023}. Apart from the heterogeneous interactions between agents, another novelty of our framework is the extension of the time horizon to infinity. To the best of the authors' knowledge, there is no functional convex order results established for such dynamics existing in the literature.

Building on the graphon theories developed by Lov\'{a}sz \cite{Lovasz2012}, the study on graphon mean-field systems has gained considerable attention recently. They play a crucial role in modeling heterogeneous network structures of many systems beyond the scope of classical mean-field models. It is well known that the graphon serves as a natural continuum limit of the graph sequence. Motivated by this, \cite{BayrWu2022,BayrKim2022,BayrChak2023,Pham2025,He2025} quantitatively studied the convergence of finite particle systems towards their graphon counterparts by establishing results such as propagation of chaos (POC) \cite{BayrWu2022,Pham2025}, law of large numbers (LLN) \cite{BayrWu2022,BayrChak2023,Pham2025,He2025}, uniform-in-time convergence \cite{BayrWu2022}, and concentration bounds on the Wasserstein distance \cite{BayrKim2022}. Moreover, Caines and Huang \cite{Caines2021} incorporated graphon into the mean-field game (MFG) framework, proved a well-posedness result for the graphon mean-field games (GMFGs), and in particular, provided an $\epsilon$-Nash result which relates the mean-field equilibrium to the equilibria of $N$ players on the graph sequence. We should also mention that an approximate Nash-equilibrium theory for GMFG was established in \cite{LackerSoret2023} under a label-state formulation, which defines different strengths of approximation under different regularities of the model.  
We refer to \cite{Aurell2022,Zhang2023,Tangpi2024} for other works on GMFG, and \cite{Crescenzo2026, Djete2025} for works on graphon mean-field type control (MFC) problems.

Our first contribution is to establish convex order preservation for graphon mean-field systems. This requires going beyond the existing results in two directions. On the one hand, the functional convex order results of \cite{LiuPages2023} concern McKean--Vlasov dynamics and do not account for heterogeneous graphon interactions. On the other hand, the uniform-in-time graphon approximation results of \cite{BayrWu2022} are established for finite particle systems with linear interactions in the drift, whereas convex order propagation in our setting requires graphon-dependent interactions in the diffusion coefficient.

A key difficulty is that the usual particle approximation cannot be used to transfer convex order to the graphon limit, since the particle system does not in general preserve convex order; see \cite[Remark 5.1]{LiuPages2023}. We therefore construct and analyze Euler schemes directly at the level of the graphon mean-field system. Under a suitable dissipativity condition, we obtain approximation estimates that are uniform in both time and the agent label. Combining these estimates with forward-backward induction arguments yields the marginal and functional convex order comparisons.

Our second contribution is to extend functional convex order to the infinite-horizon setting. This extension is not immediate from the finite-horizon theory, as the relevant moment and approximation estimates need to remain stable as the time horizon grows. The dissipativity condition provides the required uniform-in-time stability, while the exponentially weighted space $L^2_\mathcal{F} (0, \infty;\mathbb{R}^d)$, defined as the set of all $\{\mathcal{F}_t\}_{t \geq 0}$-progressively measurable processes $X := \{X_t\}_{t \geq 0}$ such that
\begin{equation*}
    \mathbb{E}\Big[\big\|X\big\|_K^2\Big] := \mathbb{E}\bigg[\int_0^\infty \mathrm{e}^{-Kt} \big|X_t\big|^2 \mathrm{d}t\bigg] < \infty,
\end{equation*}
 allows us to control entire trajectories on $[0,\infty)$. This leads to an infinite-horizon trajectory-level convergence result for the Euler schemes and, ultimately, to the functional convex order on the infinite horizon. See also  \cite{Zhang20232}.

As an application, we use the resulting convex order comparisons to compare value functions in one-dimensional linear-quadratic graphon mean-field games.



The rest of the paper is organized as follows. We conclude Section \ref{sec:intro} with notations used throughout the paper. In Section \ref{sec:formulation}, we introduce the graphon mean-field dynamics and the associated well-posedness and Euler schemes. In Section \ref{sec:main}, we provide all main results on convex ordering, namely the convex ordering of the Euler schemes and the functional convex ordering. Proofs related to well-posedness and Euler schemes are collected in Section \ref{A:sec:1}. Finally, in Section \ref{sec:application}, we apply our results to the comparison of the value functions of the graphon mean-field games.

\subsection{Notations}
1. Let $I := [0, 1]$. The set $I$ represents the labels of a continuum of agents. A graphon $G$ is a symmetric measurable function from $I \times I$ to $I$, associated with the cut-norm defined as:
$$
\|G\|_\square := \sup_{S,T \in \mathcal{B}(I)} \bigg|\int_{S \times T} G(u,v) \mathrm{d}u \mathrm{d}v\bigg|.
$$
2. We denote $\mathcal{P} (E)$, $\mathcal{P}_p (E)$, $\mathcal{M}(E)$, and $\mathcal{M}_p (E)$ respectively the collection of probability measures, probability measures with finite $p$-th moment, finite positive measures, and finite positive measures with finite $p$-th moment on $(E, |\cdot|_E)$. Also, we define the space $\mathcal{H}(E) \subset \mathcal{P}_2 (E)^I$ by
\begin{align*}
    \mathcal{H} (E) := \bigg\{ & \boldsymbol{\mu} \in \mathcal{P}_2 (E)^I \text{ s.t. } \forall B \in \mathcal{B}(E), \; \boldsymbol{\mu}(B) : u \mapsto \mu^u(B) \text{ is  measurable}  \\
    & \text{and s.t. } \sup_{u \in I} \int_E |x|_E^2 \mu^u(\mathrm{d}x) < \infty\bigg\}.
\end{align*}
3. The $p$-Wasserstein distance $\mathcal{W}_p$ on $\mathcal{P}_p (E)$ is defined for any $\mu,\nu \in \mathcal{P}_p(E)$:
$$
\mathcal{W}_p (\mu, \nu) := \inf_{\pi \in \Pi(\mu, \nu)} \left(\int_{E \times E} |x - y|_E^p \pi (\mathrm{d} x, \mathrm{d} y)\right)^{1/p},
$$
where $\Pi(\mu,\nu)$ denotes the set of probability measures in $\mathcal{P}(E \times E)$ with the first marginal $\mu$ and the second marginal $\nu$. 

Fix an arbitrary reference point $x_0 \in E$, the 2-Wasserstein On Positive measures ($\mathcal{WOP}_2$) metric, firstly introduced in \cite{Leblanc2023}, is defined for any $\mu, \nu \in \mathcal{M}_2 (E)$ by
\begin{align*}
    \mathcal{WOP}_2^2 (\mu, \nu) & = (m_\mu - m_\nu)^2 + \mathcal{W}_2^2 (T_{m_\mu} \# \overline{\mu}, T_{m_\nu} \# \overline{\nu}) \\
    & = (m_\mu - m_\nu)^2 + (m_\mu - m_\nu) (M_{x_0}(\mu) - M_{x_0}(\nu)) + m_\mu m_\nu \mathcal{W}_2^2(\overline{\mu}, \overline{\nu}),
\end{align*}
where $m_\mu$ denotes the total mass of measure $\mu$; $\overline{\mu} := \mu/m_\mu$ if $m_\mu > 0$ and $\delta_{x_0}$ otherwise, where $\delta_{x_0}$ is the Dirac measure at $x_0$; $M_{x_0}(\mu) := \int_{E} |x-x_0|_E^2 \mu(\mathrm{d}x)$; for $a >0$, we define $T_a(x) := a(x-x_0) + x_0$, and $T \# \mu$ represents the pushforward of a measure $\mu$ by a measurable mapping $T$. 

The space $\mathcal{H}(E)$ is a metric space itself when endowed with the metric defined for any $\boldsymbol{\mu}, \boldsymbol{\nu} \in \mathcal{P}_2(E)^I$ by
\begin{equation}
    \boldsymbol{d}_\mathcal{H} (\boldsymbol{\mu}, \boldsymbol{\nu}) := \sup_{u \in I} \mathcal{W}_2 (\mu^u, \nu^u) = \sup_{u \in I} \bigg(\inf_{\pi \in \Pi(\mu^u, \nu^u)} \int_{E \times E} |x-y|_E^2 \pi(\mathrm{d}x, \mathrm{d}y)\bigg)^{1/2}.
    \label{A:eqn:50}
\end{equation}
4. Let $\mathbb{M}^{d \times q} (\mathbb{R})$ denote the set of matrices with $d$ rows and $q$ columns equipped with the operator norm $\|\cdot\|$ defined by $\| A\| := \sup_{|z|_q \leq 1} |Az|_d$, for any $A \in \mathbb{M}^{d \times q}(\mathbb{R})$, where $|\cdot|_d$ denotes the canonical Euclidean norm on $\mathbb{R}^d$ generated by the canonical inner product $\langle \cdot|\cdot \rangle$.  Without ambiguity, we will write $|\cdot|$ instead of $|\cdot|_d$ for the Euclidean norm on $\mathbb{R}^d$ to simplify notation.

As (1.5) in \cite{LiuPages2023}, we define a partial order between two matrices in $\mathbb{M}^{d \times q} (\mathbb{R})$ as follows:
\begin{equation}
    \forall A, B \in \mathbb{M}^{d \times q}(\mathbb{R}), \quad A \preceq B \quad \text{if $BB^\top - A A^\top$ is a positive semi-definite (PSD) matrix},
    \label{A:eqn:7}
\end{equation}
where $A^\top$ stands for the transpose of $A$. \\
5. Given a graphon $G$ and a $\boldsymbol{\mu} \in \mathcal{H}(E)$, we define the map $I \ni u \mapsto [G \boldsymbol{\mu}]^u \in \mathcal{M}_2(E)$ by
\begin{equation*}
    [G \boldsymbol{\mu}]^u (\mathrm{d}x) := \int_I G(u,v) \mu^v(\mathrm{d}x) \mathrm{d}v.
\end{equation*}
Similar to (3.2) in \cite{Pham2025}, we have for any $\boldsymbol{\mu}, \boldsymbol{\nu} \in \mathcal{H}(E)$ that
\begin{equation}
    \sup_{u \in I} \mathcal{WOP}^2_2 ([G \boldsymbol{\mu}]^u, [G \boldsymbol{\nu}]^u) \leq \int_I \mathcal{W}_2^2 (\mu^u, \nu^u) \mathrm{d}u \leq \sup_{u \in I} \mathcal{W}^2_2(\mu^u, \nu^u).
    \label{A:eqn:22}
\end{equation}
6. Denote by $\mathcal{C}_\infty (E) := \mathbb{C} ([0, \infty), E)$ the space of continuous functions from $[0, \infty)$ to $E$, and similarly, denote by $\mathcal{C}_\infty (\mathcal{H}(E)) := \mathbb{C}([0, \infty), \mathcal{H} (E))$ the space of continuous functions from $[0, \infty)$ to $\mathcal{H}(E)$. In the case of $E =\mathbb{R}^d$, we simply write $\mathcal{C}^d_\infty$ instead of $\mathcal{C}_\infty (\mathbb{R}^d)$. \\ 
7. Let $L^2(0, \infty;\mathbb{R}^d)$ the space of measurable functions $x :=\{x_t\}_{t \geq 0}$ equipped with the weighted $L^2$-norm $\|\cdot\|_K$, which is defined for some $K \in (0, \infty) := \mathbb{R}^*_+$ as:
\begin{equation}
    \|x\|^2_K := \int_0^\infty \mathrm{e}^{-Kt} |x_t|^2 \mathrm{d}t.
    \label{A:eqn:49}
\end{equation}
8. Let $L^{2,c}(0, \infty; \mathcal{H}(\mathbb{R}^d))$ the space of  continuous functions $\boldsymbol{\mu} := \{\boldsymbol{\mu}_t\}_{t \geq 0}$ s.t.
$
\sup_{t} \boldsymbol{d}_\mathcal{H} (\boldsymbol{\mu}_t, \boldsymbol{\delta}_0) < \infty,
$
where $\boldsymbol{\delta}_0 := \{\delta^u_0\}_{u \in I}$ and $\delta^u_0 = \delta_0$, $\forall u \in I$. Moreover, we equip $ L^{2,c}(0, \infty; \mathcal{H}(\mathbb{R}^d))$ with  the  metric $\boldsymbol{d}_\mathcal{C}$ defined as:
\begin{equation}
    \boldsymbol{d}_\mathcal{C} (\boldsymbol{\mu}, \boldsymbol{\nu}) := \sup_{t\geq 0} \boldsymbol{d}_\mathcal{H}(\boldsymbol{\mu}_t, \boldsymbol{\nu}_t).
    \label{A:eqn:12}
\end{equation}
9. Let $\mathbb{C}_{cv} (E_1, E_2) := \{\varphi: E_1 \rightarrow E_2 \text{ convex function}\}$. \\
10. For any $m_1, m_2 \in \mathbb{N}$ with $m_1 \leq m_2$, we denote by $x_{m_1 : m_2} := (x_{m_1}, \cdots, x_{m_2})$, and similarly, $\boldsymbol{\mu}_{m_1:m_2} := (\boldsymbol{\mu}_{m_1}, \cdots, \boldsymbol{\mu}_{m_2})$. Moreover, we denote by $x_{0:\infty} := (x_0, x_1, x_2, \cdots)$, and $\boldsymbol{\mu}_{0:\infty} := (\boldsymbol{\mu}_{0}, \boldsymbol{\mu}_{1}, \boldsymbol{\mu}_{2}, \cdots)$. 

\section{Well-posedness and Euler schemes for graphon mean-field systems} \label{sec:formulation}
\subsection{Assumptions and well-posedness}
Given two graphons $G$ and $K$, on some filtered probability space $(\Omega, \mathcal{F}, \mathbb{F}, \mathbb{P})$, consider for any $u \in I$ the following two graphon mean-field systems:
\begin{align}
    & X^u_t = X^u_0 + \int_0^t b(s, X^u_s) \mathrm{d} s  + \int_0^t \sigma (s, X^u_s, [G \boldsymbol{\mu}_s]^u) \mathrm{d}B^u_s, \quad t \geq 0, 
    \label{A:eqn:1} \\
    & Y^u_t = Y^u_0 + \int_0^t b(s, Y^u_s) \mathrm{d} s  + \int_0^t \theta (s, Y^u_s, [K \boldsymbol{\nu}_s]^u)\mathrm{d}B^u_s, \quad t \geq 0,
    \label{A:eqn:2}
\end{align}
where we denote $\mu^u_t := \mathcal{L}(X^u_t) \in \mathcal{P}_2(\mathbb{R}^d)$ and $\boldsymbol{\mu}_t := \{\mu^u_t\}_{u \in I} \in \mathcal{P}_2(\mathbb{R}^d)^I$, and similarly, $\nu^u_t := \mathcal{L} (Y^u_t) \in \mathcal{P}_2(\mathbb{R}^d)$ and $\boldsymbol{\nu}_t := \{\nu^u_t\}_{u \in I} \in \mathcal{P}_2 (\mathbb{R}^d)^I$; $b \in \mathbb{R}^d$ and $\sigma, \theta \in \mathbb{M}^{d \times q} (\mathbb{R})$ are suitable functions; $\{B^u: u \in I\}$ are i.i.d. $q$-dimensional Brownian motions, and $\{X^u_0, Y^u_0, B^u: u \in I \}$ are mutually independent.

In the following, we first impose the following regularity and stability conditions on the coefficients in Assumption \ref{A:assump:1}. Then we introduce the structural conditions needed for the convex ordering in Assumption \ref{A:assump:2}.

\begin{assumption} 
    \textup{(i)} The functions $b, \; \sigma$ and $\theta$ are $\rho$-H\"older continuous in $t$, $\rho \in (0, 1]$, and Lipschitz continuous in $x$ and in $\mu$ in the following sense: there exists a positive constant $\tilde{L}$ such that $\forall x \in \mathbb{R}^d, \; \forall \mu \in \mathcal{M}_2(\mathbb{R}^d)$
    \begin{align*}
        & |b(t,x) - b(s,x)| \vee \|\sigma(t,x,\mu) - \sigma(s,x,\mu)\| \vee \|\theta(t,x,\mu) - \theta(s,x,\mu)\| \\
        & \qquad \leq \tilde{L} \big(1 + |x| + \mathcal{WOP}_2(\mu, \overline{\mu} \cdot \delta_{0})\big) (t-s)^\rho;
    \end{align*}
    for every fixed $t \in \mathbb{R}_+$, there exists $L > 0$ such that $\forall x,y \in \mathbb{R}^d, \; \forall \mu, \nu \in \mathcal{M}_2 (\mathbb{R}^d)$
    \begin{align*}
        |b(t,x) - b(t,y)| \vee \|\sigma(t,x,\mu) - \sigma(t,y,\nu)\| \vee \|\theta(t,x,\mu) - \theta(t,y,\nu)\| \leq L\big(|x-y|+\mathcal{WOP}_2(\mu, \nu)\big).
    \end{align*}
    \textup{(ii)} The map $I \ni u \mapsto \mu^u_0 \in \mathcal{P}(\mathbb{R}^d)$ is measurable, and $\sup_{u \in I} \mathbb{E} [|X^u_0|^2] < \infty$ $(\textit{idem for }\{\nu^u_0\}_{u \in I})$.
    
    \vspace{1mm}
    
    \noindent \textup{(iii)} Dissipativity: There exists some $c_0 \in (0,\infty)$ such that 
    \begin{equation}
        (x_1 - x_2) \cdot(b(t,x_1) - b(t, x_2)) \leq -c_0 |x_1-x_2|^2, \quad \forall t \in \mathbb{R}_+, \; \forall x_1, x_2 \in \mathbb{R}^d,
        \label{A:eqn:20}
    \end{equation}
    and
    \begin{equation}
        \kappa_1 := c_0 - 8L^2 q > 0.
        \label{A:eqn:26}
    \end{equation}
    \textup{(iv)} In addition, suppose that 
    \begin{equation}
        \kappa_2 := \sup_{t \geq 0}\big|b(t,0)\big| \vee \sup_{t \geq 0} \sup_{u \in I} \big\|\sigma(t,0,[G \boldsymbol{\delta}_0]^u)\big\| \vee \sup_{t \geq 0} \sup_{u \in I} \big\|\theta(t,0,[K \boldsymbol{\delta}_0]^u)\big\| < \infty,
        \label{A:eqn:59}
    \end{equation}
    where $\boldsymbol{\delta}_0 := \{\delta^u_0\}_{u \in I}$ and $\delta^u_0 = \delta_0$, $\forall u \in I$.
    \label{A:assump:1}
\end{assumption}

\begin{remark}
Dissipativity conditions similar to Assumption 2.1(iii) have been used in  \cite{BayrWu2022} to obtain uniform-in-time estimates for graphon particle systems. In our setting, this condition likewise plays a key role in deriving uniform-in-time moment and Euler approximation estimates, which are essential for the infinite-horizon analysis.

\end{remark}

\begin{assumption} \label{assum:2.2}
    \textup{(i)} The function $b$ is affine in $x$. 

    \vspace{2mm}

    \noindent \textup{(ii)} For every fixed $t \in \mathbb{R}_+$ and $\mu \in \mathcal{M}_2(\mathbb{R}^d)$, the function $x \mapsto \sigma(t,x,\mu)$ is convex in the sense that
    \begin{equation*}
        \forall x, y \in \mathbb{R}^d, \; \forall \lambda \in [0, 1], \quad \sigma (t,\lambda x+(1-\lambda)y, \mu) \preceq \lambda \sigma(t,x,\mu) + (1-\lambda) \sigma(t,y, \mu).
    \end{equation*}
    \textup{(iii)} For every fixed $(t,u, x) \in \mathbb{R}_+ \times I \times \mathbb{R}^d$, the function $\boldsymbol{\mu} \mapsto \sigma (t,x,[G \boldsymbol{\mu}]^u)$ is nondecreasing with respect to the convex order, namely,
    \begin{equation*}
        \forall \boldsymbol{\mu}, \boldsymbol{\nu} \in \mathcal{H}(\mathbb{R}^d) \text{ satisfying } \mu^u \preceq_{cv} \nu^u, \; \forall u \in I \; \Longrightarrow \; \sigma (t,x,[G \boldsymbol{\mu}]^u) \preceq \sigma (t,x,[G \boldsymbol{\nu}]^u).
    \end{equation*}
    \textup{(iv)} For every $(t,u, x,\boldsymbol{\mu}) \in \mathbb{R}_+ \times I \times \mathbb{R}^d \times \mathcal{H}(\mathbb{R}^d)$, we have
    \begin{equation*}
        \sigma (t,x,[G\boldsymbol{\mu}]^u) \preceq \theta (t,x,[K\boldsymbol{\mu}]^u).
    \end{equation*}
    \textup{(v)} $X^u_0 \preceq_{cv} Y^u_0, \; \forall u \in I$.
    \label{A:assump:2}
\end{assumption}

\begin{example}
    Graphon mean-field systems \eqref{A:eqn:1}-\eqref{A:eqn:2} are nonlinear with respect to the measures, which is a more general setting than the scalar framework adopted in \cite{BayrWu2022, BayrChak2023, Zhang2023}. Under the linear interaction assumption with $d=q=1$, i.e.
    \begin{align*}
    & \mathrm{d}X^u_t = b(t,X^u_t)\mathrm{d}t + \int_I \int_\mathbb{R} G(u,v) \sigma (t, X^u_t,z)\mu^v_t(\mathrm{d}z)\mathrm{d}v\mathrm{d}B^u_t, \quad \text{with } \mu^u_t = \mathcal{L}(X^u_t), \\
    & \mathrm{d}Y^u_t = b(t,Y^u_t)\mathrm{d}t + \int_I \int_\mathbb{R} K(u,v) \theta (t, Y^u_t,z)\nu^v_t(\mathrm{d}z)\mathrm{d}v\mathrm{d}B^u_t, \quad \text{with } \nu^u_t = \mathcal{L}(Y^u_t),
    \end{align*}
    Assumption \ref{A:assump:2} (ii)-(iv) can be implied by the following conditions:

    \vspace{1mm}
    
    \noindent $\bullet$ \quad for every $t \in \mathbb{R}_+$, the function $\mathbb{R} \times \mathbb{R} \ni (x,z) \mapsto \sigma (t, x, z) \in \mathbb{R}_+$ is convex;

    \vspace{1mm}
    
    \noindent $\bullet$ \quad for every $(u,t,x) \in I \times \mathbb{R}_+ \times \mathbb{R}$, we have for any $(v_1, v_2, z_1, z_2) \in I^2 \times \mathbb{R}^2$,
    \begin{equation*}
        G(u,v_1) G(u,v_2) \sigma (t,x,z_1) \sigma(t,x,z_2) \leq K(u,v_1) K(u,v_2) \theta (t,x,z_1) \theta (t,x,z_2).
    \end{equation*}
    
    To see this, Assumption \ref{A:assump:2} (ii) follows directly from the first point, and (iii) is satisfied due to the non-negativity as well as the convexity of $\sigma$, i.e.
    \begin{equation*}
        0 \leq \int_I G(u,v) \bigg(\int_\mathbb{R} \sigma (t,X^u_t,z) \mu^v_t(\mathrm{d}z)\bigg)\mathrm{d}v \leq \int_I G(u,v) \bigg(\int_\mathbb{R} \sigma (t,X^u_t,z) \nu^v_t(\mathrm{d}z)\bigg)\mathrm{d}v,
    \end{equation*}
    which verifies (iii) by noticing that
    \begin{equation*}
        \bigg[\int_I G(u,v) \bigg(\int_\mathbb{R} \sigma (t,X^u_t,z) \mu^v_t(\mathrm{d}z)\bigg)\mathrm{d}v\bigg]^2 \leq \bigg[\int_I G(u,v) \bigg(\int_\mathbb{R} \sigma (t,X^u_t,z) \nu^v_t(\mathrm{d}z)\bigg)\mathrm{d}v\bigg]^2.
    \end{equation*}
    Based on the second point, a straightforward calculation shows that Assumption \ref{A:assump:2} (iv) holds.
    \label{A:example:1}
\end{example}

The following lemma gives well-posedness of systems \eqref{A:eqn:1} and \eqref{A:eqn:2}, which extends the existing results, see Proposition 2.1 in \cite{BayrChak2023} or Proposition 3.1 in \cite{Pham2025}, from finite-horizon to infinite-horizon.

 Given $\boldsymbol{\mu} \in L^{2,c}(0,\infty; \mathcal{H}(\mathbb{R}^d))$, define $\Gamma (\boldsymbol{\mu}) := \{\{\mathcal{L}(X^{u, \boldsymbol{\mu}}_t)\}_{u \in I}\}_{t \geq 0}$ with $X^{u, \boldsymbol{\mu}}_t$ the unique solution to
\begin{equation*}
    \mathrm{d}X^{u, \boldsymbol{\mu}}_t = b(t, X^{u, \boldsymbol{\mu}}_t) \mathrm{d} t  + \sigma (t, X^{u, \boldsymbol{\mu}}_t, [G \boldsymbol{\mu}_t]^u) \mathrm{d}B^u_t, \quad X^{u, \boldsymbol{\mu}}_0 = X^u_0.
\end{equation*}
Similarly, define $\Xi (\boldsymbol{\mu}) := \{\{\mathcal{L}(Y^{u, \boldsymbol{\mu}}_t)\}_{u \in I}\}_{t \geq 0}$ with $Y^{u, \boldsymbol{\mu}}_t$ the unique solution to
\begin{equation*}
    \mathrm{d}Y^{u, \boldsymbol{\mu}}_t = b(t, Y^{u, \boldsymbol{\mu}}_t) \mathrm{d} t  + \theta (t, Y^{u, \boldsymbol{\mu}}_t, [K \boldsymbol{\mu}_t]^u) \mathrm{d}B^u_t, \quad Y^{u, \boldsymbol{\mu}}_0 = Y^u_0.
\end{equation*}
Using the Lipschitz properties of $b, \sigma$ and $\theta$ as well as the dissipativity condition on $b$, one deduces by similar arguments as in \cite[Proposition 2.1]{BayrChak2023} and \cite[Section 6]{BayrWu2022} that $\Gamma (\boldsymbol{\mu}), \; \Xi(\boldsymbol{\mu}) \in L^{2,c} (0,\infty;\mathcal{H}(\mathbb{R}^d))$. The space $\left(L^{2,c} (0,\infty;\mathcal{H}(\mathbb{R}^d)),  \boldsymbol{d}_\mathcal{C} \right)$ is complete, see Remark \ref{rmk:complete}. Hence, the following Lemma \ref{A:lemma:8} together with the Banach fixed-point theorem yields the well-posedness of \eqref{A:eqn:1}.

\begin{lemma}
    Under Assumption \ref{A:assump:1}, $\Gamma$ is a contraction under the metric $\boldsymbol{d}_\mathcal{C}$ $(\text{idem for $\Xi$})$.
    \label{A:lemma:8}
\end{lemma}

\begin{remark} \label{rmk:complete}

The space $\left(L^{2,c} (0,\infty;\mathcal{H}(\mathbb{R}^d)),  \boldsymbol{d}_\mathcal{C}\right)$ is complete due to the fact that $(\mathcal{H}(\mathbb{R}^d), \boldsymbol{d}_\mathcal{H})$ is a complete metric space. To see the latter, if $\{\boldsymbol{\mu}^n\}_n$ is a Cauchy sequence in $\mathcal{H}(\mathbb{R}^d)$, then for any fixed $u \in I$, $\{\mu^{u,n}\}_n$ is also a Cauchy sequence in the complete metric space $(\mathcal{P}_2(\mathbb{R}^d), \mathcal{W}_2)$, and so it contains a limit $\mu^u$. The measurability of the map $u \rightarrow \mu^u$ comes from Theorem 4.2.2 in \cite{Dudley2018}. This in turn determines a limit $\boldsymbol{\mu}$ in $\mathcal{H}(\mathbb{R}^d)$.

\end{remark}

\subsection{Euler schemes}
In this section, we analyze the Euler schemes of systems \eqref{A:eqn:1}-\eqref{A:eqn:2} with step size $h > 0$. For $m \in \mathbb{N}$, we define $t^h_m := m \cdot h$. Without ambiguity, we will write $t_m$ instead of $t_m^h$. Let $Z^u_{m+1} := \frac{1}{\sqrt{h}} (B^u_{t_{m+1}} - B^u_{t_m})$ be i.i.d. random variables with probability distribution $\mathcal{N}(0, \mathbf{I}_q)$. The Euler schemes of equations \eqref{A:eqn:1} and \eqref{A:eqn:2} are defined by
\begin{align}
    & \bar{X}^{u,h}_{t_{m+1}} = \bar{X}^{u,h}_{t_m} + h \cdot b(t_m, \bar{X}^{u,h}_{t_m}) + \sqrt{h} \cdot \sigma(t_m, \bar{X}^{u,h}_{t_m}, [G \bar{\boldsymbol{\mu}}^h_{t_m}]^u) Z^u_{m+1}, \quad \bar{X}^{u,h}_0 = X^u_0, \label{A:eqn:3} \\
    & \bar{Y}^{u,h}_{t_{m+1}} = \bar{Y}^{u,h}_{t_m} + h \cdot b(t_m, \bar{Y}^{u,h}_{t_m}) + \sqrt{h} \cdot \theta(t_m, \bar{Y}^{u,h}_{t_m}, [K \bar{\boldsymbol{\nu}}^h_{t_m}]^u) Z^u_{m+1}, \quad \bar{Y}^{u,h}_0 = Y^u_0, \label{A:eqn:4}
\end{align}
where $\bar{\mu}^{u,h}_{t_m} := \mathcal{L}(\bar{X}^{u,h}_{t_m})$ and $\bar{\nu}^{u,h}_{t_m} := \mathcal{L}(\bar{Y}^{u,h}_{t_m})$; $\bar{\boldsymbol{\mu}}^h_{t_m}$ and $\bar{\boldsymbol{\nu}}^h_{t_m}$ are respectively the measure ensembles of $\{\bar{\mu}^{u,h}_{t_m}\}_{u \in I}$ and $\{\bar{\nu}^{u,h}_{t_m}\}_{u \in I}$. We extend the Euler schemes to continuous time as follows (denoting $\bar{X}^{u,h}$, $\bar{Y}^{u,h}$ the corresponding processes): for every $t \in [t_m, t_{m+1})$,
\begin{align}
    & \bar{X}^{u,h}_t = \bar{X}^{u,h}_{t_m} + b(t_m, \bar{X}^{u,h}_{t_m}) (t - t_m) + \sigma(t_m, \bar{X}^{u,h}_{t_m}, [G \bar{\boldsymbol{\mu}}^h_{t_m}]^u)(B^u_t - B^u_{t_m}), \label{A:eqn:30} \\
    & \bar{Y}^{u,h}_{t} = \bar{Y}^{u,h}_{t_m} + b(t_m, \bar{Y}^{u,h}_{t_m}) (t-t_m) + \theta(t_m, \bar{Y}^{u,h}_{t_m}, [K \bar{\boldsymbol{\nu}}^h_{t_m}]^u)(B^u_t - B^u_{t_m}). \label{A:eqn:31}
\end{align} 

The following Proposition \ref{A:prop:1} provides uniform-in-time estimates and weighted \(L^2\)-convergence for the Euler schemes, which will be crucial for the infinite-horizon functional convex order. It is valid for both processes $X$ and $Y$, and its proof is postponed to Section \ref{A:sec:1}.

\begin{proposition}
    Assume that Assumption \ref{A:assump:1} is in force. 
    \begin{itemize}
        \item[\textup{(i)}] There exist $h_0, C > 0$ such that
        \begin{equation*}
            \sup_{u \in I}\mathbb{E}\Big[\big\|X^u\big\|_K^2\Big] \vee \sup_{h \in (0, h_0)} \sup_{u \in I}\mathbb{E}\Big[\big\|\bar{X}^{u,h}\big\|_K^2\Big] \leq C.
        \end{equation*}
        Moreover, there exists a constant $\kappa$ such that for any $h \in (0,h_0)$,
        \begin{equation*}
            \sup_{t \geq 0} \sup_{u \in I} \mathbb{E}\Big[\big|\bar{X}^{u,h}_t - \bar{X}^{u,h}_{[t]^h}\big|^2\Big] \leq \kappa (t-[t]^h) \leq \kappa h,
        \end{equation*}
        where we used the notation $[t]^h := \lfloor \frac{t}{h}\rfloor \cdot h$.

        \item[\textup{(ii)}] Recall $h_0 > 0$ defined in \textup{(i)}. Then, there exists $\gamma >0$ such that for any $h \in (0, h_0)$,
        \begin{equation*}
            \sup_{u \in I} \mathbb{E} \Big[\big\|X^u - \bar{X}^{u,h}\big\|_K^2\Big] \leq \gamma h^{1 \wedge 2\rho}.
        \end{equation*}
    \end{itemize}
    \label{A:prop:1}
\end{proposition}

Notice from Lemma \ref{A:lemma:9} (i) that 
$
\{\bar{\mu}^{u,h}_t\}_{t \geq 0} := \{\mathcal{L}(\bar{X}^{u,h}_t)\}_{t \geq 0} \in \mathbb{C}([0,\infty), \mathcal{P}_2 (\mathbb{R}^d)).
$ In addition, we have the following property, which justifies that the graphon-weighted measure, i.e. $[G \bar{\boldsymbol{\mu}}^h_{t_m}]^u$ or $[K \bar{\boldsymbol{\nu}}^h_{t_m}]^u$, is well-defined.
\begin{lemma}
    The map $I \ni u \mapsto \{\bar{\mu}^{u,h}_t\}_{t \geq 0} :=  \{\mathcal{L}(\bar{X}^{u,h}_t)\}_{t \geq 0} \in \mathbb{C}([0,\infty), \mathcal{P}_2(\mathbb{R}^d))$ is measurable $($idem for $\{\bar{\nu}^{u,h}_t\}_{t \geq 0}$$)$.
    \label{A:lemma:7}
\end{lemma}
The proof of Lemma \ref{A:lemma:7} is also postponed to Section \ref{A:sec:1}.

\section{Main results} \label{sec:main}
\subsection{Convex order results for Euler schemes}
In this section, we will establish the convex ordering results for the random variables $\bar{X}^{u,h}_{t_m}$ and $\bar{Y}^{u,h}_{t_m}$, $m \in \mathbb{N}$, defined by the Euler schemes \eqref{A:eqn:3} and \eqref{A:eqn:4}. In order to simplify the notation, we rewrite \eqref{A:eqn:3} and \eqref{A:eqn:4} by setting
$$
\bar{X}^u_m := \bar{X}^{u,h}_{t_m}, \qquad \bar{Y}^u_m := \bar{Y}^{u,h}_{t_m}, \qquad \bar{\mu}^u_m := \bar{\mu}^{u,h}_{t_m}, \qquad \bar{\nu}^u_m := \bar{\nu}^{u,h}_{t_m}.
$$
Then, we have
\begin{align}
    & \bar{X}^u_{m+1} = b_m (\bar{X}^u_{m}) + \sigma^u_m (\bar{X}^u_m, \bar{\boldsymbol{\mu}}_m) Z^u_{m+1}, \quad \bar{X}^u_0 = X^u_0, \quad \text{where } \bar{\boldsymbol{\mu}}_m := \{\bar{\mu}^u_m\}_{u \in I}, \label{A:eqn:5}\\
    & \bar{Y}^u_{m+1} = b_m (\bar{Y}^u_{m}) + \theta^u_m (\bar{Y}^u_m, \bar{\boldsymbol{\nu}}_m) Z^u_{m+1}, \quad \bar{Y}^u_0 = Y^u_0, \quad \text{where } \bar{\boldsymbol{\nu}}_m := \{\bar{\nu}^u_m\}_{u \in I}, \label{A:eqn:6}
\end{align}
where for every $m \in \mathbb{N}$, we define:
\begin{align}
    & b_m (x) := x + h \cdot b(t_m, x), \qquad \sigma^u_m (x,\boldsymbol{{\mu}}) := \sqrt{h} \cdot \sigma (t_m, x, [G \boldsymbol{\mu}]^u), \label{A:eqn:64}\\
    & \theta^u_m (x,\boldsymbol{{\mu}}) := \sqrt{h} \cdot \theta (t_m, x, [K \boldsymbol{\mu}]^u). \label{A:eqn:65}
\end{align}
First, we notice that it follows from Lemma \ref{A:lemma:9} (i) and Lemma \ref{A:lemma:7} that $\bar{\boldsymbol{\mu}}_m \in \mathcal{H}(\mathbb{R}^d)$ for any $m =0, 1,2, \cdots$. Then, by Assumption \ref{A:assump:2}, $\bar{X}^u_0, \; \bar{Y}^u_0, \; b_m, \; \sigma^u_m \; \text{and } \theta^u_m$, for any $m \in \mathbb{N}$ and $u \in I$, satisfy the discrete-time counterpart.
\begin{assumption}
    \textup{(i)} The function $b_m$ is affine in $x$. 

    \vspace{2mm}

    \noindent \textup{(ii)} The function $\sigma^u_m$ is convex in $x$ in the sense that
    $$
    \forall x, y \in \mathbb{R}^d, \; \forall \lambda \in [0, 1], \quad \sigma^u_m (\lambda x + (1-\lambda)y, \boldsymbol{\mu}) \preceq \lambda \sigma^u_m (x, \boldsymbol{\mu}) + (1-\lambda) \sigma^u_m(y, \boldsymbol{\mu}).
    $$
    \textup{(iii)} The function $\sigma^u_m$ is nondecreasing in $\boldsymbol{\mu}$ with respect to the convex order:
    $$
    \forall \boldsymbol{\mu}, \boldsymbol{\nu} \in \mathcal{H}(\mathbb{R}^d) \text{ satisfying } \mu^u \preceq_{cv} \nu^u, \; \forall u \in I \; \Longrightarrow \; \sigma^u_m(x, \boldsymbol{\mu}) \preceq \sigma^u_m(x, \boldsymbol{\nu}).
    $$
    \textup{(iv)} We have the following order between $\sigma^u_m$ and $\theta^u_m$:
    $$
    \forall (x, \boldsymbol{\mu}) \in \mathbb{R}^d \times \mathcal{H}(\mathbb{R}^d), \quad \sigma^u_m (x, \boldsymbol{\mu}) \preceq \theta^u_m(x, \boldsymbol{\mu}).
    $$
    \textup{(v)} $\bar{X}^u_0 \preceq_{cv} \bar{Y}^u_0, \; \forall u \in I$.
    \label{A:assump:3}
\end{assumption} 

We recall that convex order can be characterized using convex functions with linear growth; see \cite[Lemma 4.1]{LiuPages2023} and \cite[Lemma A.1]{Alfonsi2019}.

\begin{lemma}
    Let $\mu, \nu \in \mathcal{P}_1 (\mathbb{R}^d)$. Then, we have $\mu \preceq_{cv} \nu$ if and only if
    $$
    \forall \varphi: \mathbb{R}^d \rightarrow \mathbb{R} \text{ convex and such that } \sup_{x \in \mathbb{R}^d} \frac{|\varphi(x)|}{1 + |x|} < \infty, \; \int_{\mathbb{R}^d} \varphi(x) \mu(\mathrm{d}x) \leq \int_{\mathbb{R}^d} \varphi(y) \nu(\mathrm{d}y).
    $$
    \label{A:lemma:2}
\end{lemma}

\subsubsection{ One-marginal case}
For every $m \in \mathbb{N}$ and $u \in I$, we define an operator $Q^u_{m+1}: \mathbb{C}_{cv}(\mathbb{R}^d, \mathbb{R}) \rightarrow \mathbb{C} (\mathbb{R}^d \times \mathbb{M}^{d\times q}(\mathbb{R}), \mathbb{R})$ associated with $Z^u_{m+1}$ by 
\begin{align}
    \mathbb{R}^d \times \mathbb{M}^{d \times q}(\mathbb{R}) \ni (x, e) \longmapsto & \big(Q^u_{m+1} \varphi \big)(x, e) := \mathbb{E} \Big[\varphi\big(b_m(x) + e Z^u_{m+1}\big)\Big] \in \mathbb{R}.
    \label{A:eqn:56}
\end{align}
For every $m \in \mathbb{N}^*$, let $\mathcal{F}_m$ denote the $\sigma$-algebra generated by $\{X^u_0, Y^u_0, Z^u_1, \cdots, Z^u_m\}_{u \in I}$. We can then obtain the following one-marginal result.

\begin{proposition}Let $\{\bar{X}^u_m\}$ and $\{\bar{Y}^u_m\}$ be random variables defined by \eqref{A:eqn:5} and \eqref{A:eqn:6}. Under Assumption \ref{A:assump:3}, we have
$$
\bar{X}^u_m \preceq_{cv} \bar{Y}^u_m, \quad \forall m \in \mathbb{N}, \; \forall u \in I.
$$
\label{A:prop:2} 
\end{proposition}

Proposition \ref{A:prop:2} relies on the following lemma whose proof is similar to \cite[Proofs of Lemmas 4.2 and 4.3]{LiuPages2023}, hence is omitted.
\begin{lemma}
    Let $\varphi \in \mathbb{C}_{cv}(\mathbb{R}^d, \mathbb{R})$ with linear growth. Then, for every $m \in \mathbb{N}^*$ and $u \in I$:
    
    \textup{(i)} the function $(x, e) \mapsto (Q^u_m \varphi) (x,e)$ is (finite and) convex; 

    \vspace{1mm}

    \textup{(ii)} for any fixed $x \in \mathbb{R}^d$, the function $e \mapsto (Q^u_m \varphi)(x, e)$ attains its minimum at $\mathbf{0}^{d \times q}$, where $\mathbf{0}^{d \times q}$ is the zero-matrix of size $d \times q$; 

    \vspace{1mm}
    
    \textup{(iii)} for any fixed $x \in \mathbb{R}^d$, the function $e \mapsto (Q^u_m \varphi)(x,e)$ is nondecreasing with respect to the partial order of $d \times q$ matrix, defined in \eqref{A:eqn:7};

    \vspace{1mm}

    \textup{(iv)} for any $\boldsymbol{\mu} \in \mathcal{H}(\mathbb{R}^d)$, the function $x \mapsto \mathbb{E}\big[\varphi\big(b_m(x)+\sigma^u_m(x, \boldsymbol{\mu}) Z^u_{m+1}\big)\big]$ is convex with linear growth.
    \label{A:lemma:1}
\end{lemma}

\begin{proof}[Proof of Proposition \ref{A:prop:2}]
    Assumption \ref{A:assump:3} directly implies $\bar{X}^u_0 \preceq_{cv} \bar{Y}^u_0$, $\forall u \in I$. Assume $\bar{X}^u_m \preceq_{cv} \bar{Y}^u_m$, $\forall u \in I$, or equivalently, $\bar{\mu}^u_m \preceq_{cv} \bar{\nu}^u_m$, $\forall u \in I$. Let $\varphi \in \mathbb{C}_{cv}(\mathbb{R}^d, \mathbb{R})$ with linear growth. Then, for any $u \in I$:
    \begin{align*}
        \mathbb{E}\Big[\varphi\big(\bar{X}^u_{m+1}\big)\Big] &= \mathbb{E} \Big[\varphi\big(b_m(\bar{X}^u_m) +\sigma^u_m(\bar{X}^u_m, \bar{\boldsymbol{\mu}}_m) Z^u_{m+1}\big)\Big] \\
        &= \mathbb{E} \Big[\mathbb{E} \Big[\varphi\big(b_m(\bar{X}^u_m) +\sigma^u_m(\bar{X}^u_m, \bar{\boldsymbol{\mu}}_m) Z^u_{m+1}\big) \Big| \mathcal{F}_m\Big]\Big] \\
        & = \int_{\mathbb{R}^d} \bar{\mu}^u_m (\mathrm{d}x) \mathbb{E} \Big[\varphi\big(b_m(x) +\sigma^u_m(x, \bar{\boldsymbol{\mu}}_m) Z^u_{m+1}\big)\Big] \\
        & \quad \text{(the integrability is due to Lemma \ref{A:lemma:9} (i) and Lemma \ref{A:lemma:1} (iv))} \\
        & \leq \int_{\mathbb{R}^d} \bar{\mu}^u_m (\mathrm{d}x) \mathbb{E} \Big[\varphi\big(b_m(x) +\sigma^u_m(x, \bar{\boldsymbol{\nu}}_m) Z^u_{m+1}\big)\Big] \\
        & \quad \text{(by Assumption \ref{A:assump:3} (iii) and Lemma \ref{A:lemma:1} (iii), since $\bar{\mu}^u_m \preceq_{cv} \bar{\nu}^u_m$, $\forall u \in I$)} \\
        & \leq \int_{\mathbb{R}^d} \bar{\nu}^u_m (\mathrm{d}x) \mathbb{E} \Big[\varphi\big(b_m(x) +\sigma^u_m(x, \bar{\boldsymbol{\nu}}_m) Z^u_{m+1}\big)\Big] \\
        & \quad \text{(by Lemma \ref{A:lemma:1} (iv), since $\bar{\mu}^u_m \preceq_{cv} \bar{\nu}^u_m$)} \\
        & \leq \int_{\mathbb{R}^d} \bar{\nu}^u_m (\mathrm{d}x) \mathbb{E} \Big[\varphi\big(b_m(x) +\theta^u_m(x, \bar{\boldsymbol{\nu}}_m) Z^u_{m+1}\big)\Big] \\
        & \quad \text{(by Assumption \ref{A:assump:3} (iv) and Lemma \ref{A:lemma:1} (iii))} \\
        &= \mathbb{E} \Big[\varphi\big(\bar{Y}^u_{m+1}\big)\Big].
    \end{align*}
    Thus, $\bar{X}^u_{m+1} \preceq_{cv} \bar{Y}^u_{m+1}$ by applying Lemma \ref{A:lemma:2}. One concludes the proof by a forward induction.
\end{proof}

\subsubsection{Multi-marginal case}

We next extend the one-marginal comparison to convex functionals depending on multiple time marginals. To this end, we use a backward induction argument based on the following recursively defined functions.

Fix $T > 0$ and define $M := \frac{T}{h}$ if $\frac{T}{h} \in \mathbb{N}$, otherwise $M := \lfloor \frac{T}{h}\rfloor +1$. For every $u \in I$, we recursively define a sequence of functions
$$
\Phi^u_m: (\mathbb{R}^d)^{m+1} \times (\mathcal{P}_2 (\mathbb{R}^d)^I)^\infty \longrightarrow \mathbb{R}, \quad m = \mathbb{N} \cup \{\infty\},
$$
in a backward way as follows: \\ 
$\bullet$ \quad Set for $m \geq M$ or $m = \infty$:
\begin{equation}
    \Phi^u_m \big(x_{0:M-1}, \underbrace{x_M, \cdots,x_M}_{\text{\#: $m-M+1$}}; \boldsymbol{\mu}_{0:\infty}\big) = F\big(x_{0:M-1}, \underbrace{x_M,x_M, \cdots}_{\text{\#: $\infty$}}\big),
    \label{A:eqn:13}
\end{equation}
where $F: (\mathbb{R}^d)^{\infty} \rightarrow \mathbb{R}$ is a convex function with quadratic growth in the sense that
\begin{equation}
    \exists C >0, \; \forall x:=x_{0:\infty} \in (\mathbb{R}^d)^\infty \text{ such that } |F(x)| \leq C\Big(1+\|x\|^2_K\Big),
    \label{A:eqn:54}
\end{equation}
where $\|x\|^2_K$ is defined as follows:
\begin{equation*}
    \|x\|^2_K:= \sum_{m=0}^\infty \mathrm{e}^{-K t_m} \frac{|x_m|^2 + |x_{m+1}|^2}{2} \Delta t_m = \sum_{m=0}^\infty \mathrm{e}^{-K t_m} \frac{|x_m|^2 + |x_{m+1}|^2}{2} h
\end{equation*}
with $t_m := m\cdot h$ and $\Delta t_m := t_{m+1}-t_m$. \\
$\bullet$ \quad Set for $m = M-1$: 
\begin{equation}
    \Phi^u_m\big(x_{0:m}; \boldsymbol{\mu}_{0:\infty}\big) = \mathbb{E} \Big[\Phi^u_{m+1} \big(x_{0:m}, b_m(x_m) + \sigma^u_m (x_m, \boldsymbol{\mu}_m)Z^u_{m+1}; \boldsymbol{\mu}_{0:\infty}\big)\Big],
    \label{A:eqn:66}
\end{equation}
where $b_m$, $\sigma^u_m$ and $Z^u_{m+1}$ are given in \eqref{A:eqn:64} and \eqref{A:eqn:65} if $T/h \in \mathbb{N}$, otherwise
\begin{align}
    & b_m(x) := x + (T-t_{M-1}) \cdot b(t_m,x), \quad  Z^u_T := \frac{B^u_T -B^u_{t_{M-1}}}{\sqrt{T-t_{M-1}}} \sim \mathcal{N}(0, \boldsymbol{1}_q), \label{A:eqn:67} \\
    & \sigma^u_m(x, \boldsymbol{\mu}) := \sqrt{T-t_{M-1}} \cdot \sigma(t_m, x, [G \boldsymbol{\mu}]^u).\label{A:eqn:68}
\end{align}

$\bullet$ \quad Set for $m < M-1$:
\begin{align}
    \Phi^u_m\big(x_{0:m}; \boldsymbol{\mu}_{0:\infty}\big) &=  \Big(Q^u_{m+1} \Phi^u_{m+1} \big(x_{0:m}, \cdot; \boldsymbol{\mu}_{0:\infty}\big)\Big)\big(x_m, \sigma^u_m (x_m, \boldsymbol{\mu}_m)\big) \nonumber \\
    & = \mathbb{E} \Big[\Phi^u_{m+1} \big(x_{0:m}, b_m(x_m) + \sigma^u_m (x_m, \boldsymbol{\mu}_m)Z^u_{m+1}; \boldsymbol{\mu}_{0:\infty}\big)\Big]
    \label{A:eqn:14}
\end{align}
 with $b_m$, $\sigma^u_m$ and $Z^u_{m+1}$ defined as usual.

The functions $\{\Phi^u_m\}$ share the following properties.

\begin{lemma}[Lemma 4.4 in \cite{LiuPages2023}]
    For every $m \in \mathbb{N} \cup \{\infty\}$ and $u \in I$: 
    
    \textup{(i)} For a fixed $\boldsymbol{\mu}_{0:\infty} \in (\mathcal{P}_2(\mathbb{R}^d)^I)^\infty$, the function $\Phi^u_m(\cdot; \boldsymbol{\mu}_{0:\infty})$ is convex and has a quadratic growth in $x_{0:m}$ so that $\Phi^u_m$ is well-defined.

    \vspace{1mm}

    \textup{(ii)} For a fixed $x_{0:m} \in (\mathbb{R}^d)^{m+1}$, the function $\Phi^u_m(x_{0:m}; \cdot)$ is nondecreasing in $\boldsymbol{\mu}_{0:\infty}$ with respect to the convex order in the sense that for any $\boldsymbol{\mu}_{0:\infty}, \boldsymbol{\nu}_{0:\infty} \in (\mathcal{P}_2(\mathbb{R}^d)^I)^\infty$ with $\mu^u_j \preceq_{cv} \nu^u_j, \; \forall j \in \mathbb{N}, \; \forall u \in I$,
    $$
    \Phi^u_m \big(x_{0:m}; \boldsymbol{\mu}_{0:\infty}\big) \leq \Phi^u_m \big(x_{0:m}; \boldsymbol{\nu}_{0:\infty}\big).
    $$
    \label{A:lemma:4}
\end{lemma}

\begin{proof}
    \textup{(i)} For every $m \geq M$ or $m= \infty$, the function $\Phi^u_m$ is convex due to the convexity of $F$. Suppose that the map $x_{0:m+1} \mapsto \Phi^u_{m+1}(x_{0:m+1}; \boldsymbol{\mu}_{0:\infty})$ is convex for some $m \leq M-1$. For any $x_{0:m}, y_{0:m} \in (\mathbb{R}^d)^{m+1}$ and $\lambda \in [0, 1]$, we have
    \begin{align*}
        & \Phi^u_m \big(\lambda x_{0:m} + (1-\lambda) y_{0:m}; \boldsymbol{\mu}_{0:\infty}\big) \\
        & \quad = \mathbb{E} \Big[\Phi^u_{m+1} \big( \lambda x_{0:m} + (1-\lambda) y_{0:m}, b_m\big(\lambda x_m + (1-\lambda)y_m\big) \\
        & \qquad + \sigma^u_m \big(\lambda x_m + (1-\lambda) y_m, \boldsymbol{\mu}_m\big)Z^u_{m+1}; \boldsymbol{\mu}_{0:\infty}\big)\Big] \\
        & \quad \leq \mathbb{E} \Big[\Phi^u_{m+1} \big( \lambda x_{0:m} + (1-\lambda) y_{0:m}, \lambda b_m( x_m) + (1-\lambda) b_m(y_m) \\
        & \qquad + \big(\lambda \sigma^u_m \big(x_m, \boldsymbol{\mu}_m\big)+(1-\lambda) \sigma^u_m \big(y_m, \boldsymbol{\mu}_m\big)\big)Z^u_{m+1}; \boldsymbol{\mu}_{0:\infty}\big)\Big] \\
        & \qquad \text{(by Assumption \ref{A:assump:3} (ii) and Lemma \ref{A:lemma:1} (iii), since $\Phi^u_{m+1} (x_{0:m},\cdot; \boldsymbol{\mu}_{0:\infty})$ is convex)} \\
        & \quad \leq \lambda \mathbb{E} \Big[\Phi^u_{m+1} \big(x_{0:m}, b_m( x_m)  +  \sigma^u_m \big(x_m, \boldsymbol{\mu}_m\big)Z^u_{m+1}; \boldsymbol{\mu}_{0:\infty}\big)\Big] \\
        & \qquad + (1-\lambda) \mathbb{E} \Big[\Phi^u_{m+1} \big(y_{0:m}, b_m(y_m)  +  \sigma^u_m \big(y_m, \boldsymbol{\mu}_m\big)Z^u_{m+1}; \boldsymbol{\mu}_{0:\infty}\big)\Big] \\
        & \quad = \lambda \Phi^u_m \big(x_{0:m}; \boldsymbol{\mu}_{0:\infty}\big) + (1-\lambda)\Phi^u_m \big(y_{0:m}; \boldsymbol{\mu}_{0:\infty}\big).
    \end{align*}
    Thus, the function $\Phi^u_m (\cdot; \boldsymbol{\mu}_{0:\infty})$ is convex, and one completes the proof by a backward induction.
    
    For every $m \geq M$ or $m= \infty$, it is obvious that $\Phi^u_m$ has a quadratic growth in the sense of \eqref{A:eqn:54}. Suppose that $\Phi^u_{m+1}$ has a quadratic growth for some $m \leq M-1$. Then, $\Phi^u_m$ also has a quadratic growth as $b_m$ and $\sigma^u_m$ have a linear growth by Assumption \ref{A:assump:1}, and one concludes by a backward induction. 
    
    \textup{(ii)} Firstly, notice that for any $\boldsymbol{\mu}_{0:\infty}, \boldsymbol{\nu}_{0:\infty} \in (\mathcal{P}_2(\mathbb{R}^d)^I)^\infty$ with $\mu^u_j \preceq_{cv} \nu^u_j, \; \forall j \in \mathbb{N}, \; \forall u \in I$, we have for $m \geq M$ or $m= \infty$,
    \begin{equation*}
        \Phi^u_m \big(x_{0:M-1}, \underbrace{x_M, \cdots,x_M}_{\text{\#: $m-M+1$}}; \boldsymbol{\mu}_{0:\infty}\big) = F\big(x_{0:M-1}, \underbrace{x_M,x_M, \cdots}_{\text{\#: $\infty$}}\big) = \Phi^u_m \big(x_{0:M-1}, \underbrace{x_M, \cdots,x_M}_{\text{\#: $m-M+1$}}; \boldsymbol{\nu}_{0:\infty}\big).
    \end{equation*}
    Next assume that $\Phi^u_{m+1} (x_{0:m+1}; \cdot)$ is nondecreasing with respect to the convex order of $\boldsymbol{\mu}_{0:\infty}$ for some $m \leq M-1$. Then, we obtain that
    \begin{align*}
        \Phi^u_m \big(x_{0:m}; \boldsymbol{\mu}_{0:\infty}\big) & = \mathbb{E} \Big[\Phi^u_{m+1} \big(x_{0:m}, b_m(x_m)+\sigma^u_m (x_m, \boldsymbol{\mu}_m)Z^u_{m+1}; \boldsymbol{\mu}_{0:\infty}\big)\Big] \\
        & \leq  \mathbb{E} \Big[\Phi^u_{m+1} \big(x_{0:m}, b_m(x_m)+\sigma^u_m (x_m, \boldsymbol{\nu}_m)Z^u_{m+1}; \boldsymbol{\mu}_{0:\infty}\big)\Big] \\
        & \quad \text{(by Assumption \ref{A:assump:3} and Lemma \ref{A:lemma:1}, since $\Phi^u_{m+1}(x_{0:m}, \cdot; \boldsymbol{\mu}_{0:\infty})$ is convex)} \\
        & \leq \mathbb{E} \Big[\Phi^u_{m+1} \big(x_{0:m}, b_m(x_m)+\sigma^u_m (x_m, \boldsymbol{\nu}_m)Z^u_{m+1}; \boldsymbol{\nu}_{0:\infty}\big)\Big] =  \Phi^u_m \big(x_{0:m}; \boldsymbol{\nu}_{0:\infty}\big).
    \end{align*}
    One completes proof by a backward induction. 
\end{proof}

As $F$ has a quadratic growth in the sense of \eqref{A:eqn:54}, the integrability of $F(\bar{X}^u_{t_0:t_{M-1}}, \bar{X}^u_{T}, \bar{X}^u_{T}, \cdots)$ and $F(\bar{Y}^u_{t_0:t_{M-1}}, \bar{Y}^u_{T}, \bar{Y}^u_{T}, \cdots)$ is guaranteed by Lemma \ref{A:lemma:9} (i). To see this, for $h \in (0,h_0)$,
\begin{align*}
    \mathbb{E} \Big[\big|F\big(\bar{X}^u_{t_0:t_{M-1}}, \bar{X}^u_{T}, \bar{X}^u_{T}, \cdots\big)\big|\Big] & \leq C\left(1+ \frac{1}{2}\sum_{i=0}^\infty \mathrm{e}^{-K h i} \Big(\mathbb{E}\Big[\big|\bar{X}^u_{t_i \wedge T}\big|^2\Big] + \mathbb{E}\Big[\big|\bar{X}^{u}_{t_{i+1} \wedge T}\big|^2\Big]\Big) h\right) \\
    & \leq C \bigg(1 + \frac{\widehat{C}h}{1-\mathrm{e}^{-Kh}}\bigg) \\
    &  \rightarrow C \bigg(1 + \frac{\widehat{C}}{K}\bigg), \quad \text{as $h\rightarrow 0$}.
\end{align*}
Then, we define $\{\mathcal{X}^u_m\}$ as follows:
$$
\mathcal{X}^u_m := \mathbb{E}\Big[F\big(\bar{X}^u_{t_0:t_{M-1}}, \bar{X}^u_{ T}, \bar{X}^u_{T}, \cdots\big)\Big|  \mathcal{F}_{t_m \wedge T}\Big], \quad \text{for any $m \in \mathbb{N} \cup \{\infty\}$ and  $u \in I$,} 
$$
 where $\mathcal{F}_t := \{X^u_0, Y^u_0, Z^u_1, \cdots, Z^u_{\lfloor \frac{t}{h}\rfloor}, Z^u_t\}_{u \in I}$ with $Z^u_t := (B^u_t - B^u_{[t]^h})/(t - [t]^h)$. Recall that $\bar{\mu}^u_m := \mathcal{L}(\bar{X}^u_m)$ and $\bar{\boldsymbol{\mu}}_m := \{\bar{\mu}^u_m\}_{u \in I}$.

\begin{lemma}
    For every $m \in \mathbb{N} \cup \{\infty\}$ and $u \in I$, we have
    $
    \Phi^u_m \big(\bar{X}^u_{t_0 \wedge T:t_m \wedge T}; \bar{\boldsymbol{\mu}}_{0:\infty}\big) = \mathcal{X}^u_m.
    $
    \label{A:lemma:3}
\end{lemma}

\noindent The proof is similar to that of Lemma 4.5 in \cite{LiuPages2023}, but we provide one here for completeness.

\begin{proof}[Proof of Lemma \ref{A:lemma:3}]
     The proof follows from a backward induction. Firstly, it is obvious that for $\forall m \geq M \text{ or } m = \infty$ as
     $$
     \Phi^u_m \big(\bar{X}^u_{t_0:t_{M-1}},\underbrace{\bar{X}^u_{T}, \cdots, \bar{X}^u_{T}}_{\text{\#: $m-M+1$}}; \bar{\boldsymbol{\mu}}_{0:\infty}\big) = F\big(\bar{X}^u_{t_0:t_{M-1}}, \underbrace{\bar{X}^u_{T}, \bar{X}^u_{T}, \cdots}_{\text{\#: $\infty$}}\big) = \mathcal{X}^u_m.
     $$
     Assume that $\Phi^u_{m+1} ( \bar{X}^u_{t_0 \wedge T:t_{m+1} \wedge T}; \bar{\boldsymbol{\mu}}_{0:\infty}) = \mathcal{X}^u_{m+1}$ holds for some $m \leq M-1$. Then, we prove for the case of $m$,
    \begin{align*}
        \mathcal{X}^u_m & = \mathbb{E} \Big[\mathcal{X}^u_{m+1} \Big|\mathcal{F}_m\Big] = \mathbb{E} \Big[\Phi^u_{m+1} \big( \bar{X}^u_{t_0 \wedge T:t_{m+1} \wedge T}; \bar{\boldsymbol{\mu}}_{0:\infty}\big) \Big|\mathcal{F}_m\Big] \\
        & = \mathbb{E} \Big[\Phi^u_{m+1} \big(\bar{X}^u_{0:m}, b_m (\bar{X}^u_{m}) + \sigma^u_m (\bar{X}^u_m, \bar{\boldsymbol{\mu}}_m) Z^u_{m+1}; \bar{\boldsymbol{\mu}}_{0:\infty}\big) \Big|\mathcal{F}_m\Big] \\
        & = \Phi^u_m \big(\bar{X}^u_{0:m}; \bar{\boldsymbol{\mu}}_{0:\infty}\big),
    \end{align*}
    where we used \eqref{A:eqn:14} and \eqref{A:eqn:66} for the last line. 
\end{proof}

Similarly, for every $u \in I$, we define
$$
\Psi^u_m : (\mathbb{R}^d)^{m+1} \times (\mathcal{P}_2(\mathbb{R}^d)^I)^\infty \longrightarrow \mathbb{R}, \quad \forall m \in \mathbb{N} \cup \{\infty\},
$$
by\\
$\bullet$ \quad Set for $m \geq M$ or $m = \infty$:
\begin{equation*}
    \Psi^u_m \big(x_{0:M-1}, \underbrace{x_M, \cdots, x_M}_{\text{\#: $m-M+1$}}; \boldsymbol{\mu}_{0:\infty}\big) = F\big({ x_{0:M-1}}, \underbrace{x_M,x_M, \cdots}_{\text{\#: $\infty$}}\big), 
\end{equation*}
where $F: (\mathbb{R}^d)^\infty \rightarrow \mathbb{R}$ is a convex function with quadratic growth in the sense of \eqref{A:eqn:54}. \\
$\bullet$ \quad Set for $m = M-1$: 
\begin{equation*}
    \Psi^u_m\big(x_{0:m}; \boldsymbol{\mu}_{0:\infty}\big) = \mathbb{E} \Big[\Psi^u_{m+1} \big(x_{0:m}, b_m(x_m) + \theta^u_m (x_m, \boldsymbol{\mu}_m)Z^u_{m+1}; \boldsymbol{\mu}_{0:\infty}\big)\Big],
\end{equation*}
where $b_m$, $\theta^u_m$ and $Z^u_{m+1}$ are given in \eqref{A:eqn:64} and \eqref{A:eqn:65} if $T/h \in \mathbb{N}$, otherwise
\begin{align}
    & b_m(x) := x + (T-t_{M-1}) \cdot b(t_m,x), \quad  Z^u_T := \frac{B^u_T -B^u_{t_{M-1}}}{\sqrt{T-t_{M-1}}} \sim \mathcal{N}(0, \boldsymbol{1}_q), \nonumber \\
    & \theta^u_m(x, \boldsymbol{\mu}) := \sqrt{T-t_{M-1}} \cdot \theta(t_m, x, [K \boldsymbol{\mu}]^u). \label{A:eqn:69}
\end{align}

$\bullet$ \quad Set for $m < M-1$:
\begin{align*}
    \Psi^u_m\big(x_{0:m}; \boldsymbol{\mu}_{0:\infty}\big) &=  \Big(Q^u_{m+1} \Psi^u_{m+1} \big(x_{0:m}, \cdot; \boldsymbol{\mu}_{0:\infty}\big)\Big)\big(x_m, \theta^u_m (x_m, \boldsymbol{\mu}_m)\big) \nonumber \\
    & = \mathbb{E} \Big[\Psi^u_{m+1} \big(x_{0:m}, b_m(x_m) + \theta^u_m (x_m, \boldsymbol{\mu}_m)Z^u_{m+1}; \boldsymbol{\mu}_{0:\infty}\big) \Big]
\end{align*}
 with $b_m$, $\theta^u_m$ and $Z^u_{m+1}$ defined as usual.

Recall that $\bar{\nu}^u_m := \mathcal{L}(\bar{Y}^u_m)$, and $\bar{\boldsymbol{\nu}}_m :=  \{\bar{\nu}^u_m\}_{u \in I}$. By a similar argument as the proof of Lemma \ref{A:lemma:3}, we obtain that $\text{for any $m \in \mathbb{N} \cup \{\infty\}$ and $u \in I$}$,
\begin{equation}
    \Psi^u_m \big(\bar{Y}^u_{t_0 \wedge T:t_m \wedge T}; \bar{\boldsymbol{\nu}}_{0:\infty}\big) = \mathcal{Y}^u_m := \mathbb{E}\Big[F\big(\bar{Y}^u_{t_0:t_{M-1}}, \bar{Y}^u_{T}, \bar{Y}^u_{T}, \cdots\big)\Big|  \mathcal{F}_{t_m \wedge T}\Big].
    \label{A:eqn:9}
\end{equation}

The main result of this section is the following multi-marginal convex ordering result.
\begin{theorem}
    Under Assumptions \ref{A:assump:1} and \ref{A:assump:2}, for every convex function $F: (\mathbb{R}^d)^\infty \rightarrow \mathbb{R}$ with quadratic growth in the sense of \eqref{A:eqn:54}, we have
    $$
    \mathbb{E}\Big[F\big(\bar{X}^u_{t_0:t_{M-1}}, \bar{X}^u_{T}, \bar{X}^u_{ T}, \cdots\big)\Big] \leq \mathbb{E}\Big[F\big(\bar{Y}^u_{t_0:t_{M-1}}, \bar{Y}^u_{T}, \bar{Y}^u_{T}, \cdots\big)\Big], \quad \forall u \in I.
    $$
    \label{A:prop:3}
\end{theorem}

\begin{proof}
    We start by proving by a backward induction that 
    \begin{equation}
        \Phi^u_m \leq \Psi^u_m, \quad \forall m \in \mathbb{N} \cup \{\infty\}, \; \forall u \in I.
        \label{A:eqn:10}
    \end{equation}
    It follows from the definition of $\{\Phi^u_m\}$ and $\{\Psi^u_m\}$ that for $m \geq M$ or $m = \infty$,
    $$
    \Phi^u_m \big(x_{0:M-1}, \underbrace{x_M,\cdots, x_M}_{\text{\#: $m-M+1$}}; \boldsymbol{\mu}_{0:\infty}\big) := F\big(x_{0:M-1}, \underbrace{x_M, x_M, \cdots}_{\text{\#: $\infty$}} \big) =: \Psi^u_m\big(x_{0:M-1}, \underbrace{x_M, \cdots, x_M}_{\text{\#: $m-M+1$}}; \boldsymbol{\mu}_{0:\infty}\big).
    $$
    Assume now $\Phi^u_{m+1} \leq \Psi^u_{m+1}$ for some $m \leq M-1$, and consider for any $x_{0:m} \in (\mathbb{R}^d)^{m+1}$ and $\boldsymbol{\mu}_{0:\infty} \in (\mathcal{P}_2(\mathbb{R}^d)^I)^\infty$:
    \begin{align*}
        \Phi^u_m\big(x_{0:m}; \boldsymbol{\mu}_{0:\infty}\big) &= \mathbb{E} \Big[\Phi^u_{m+1} \big(x_{0:m}, b_m(x_m) + \sigma^u_m (x_m, \boldsymbol{\mu}_m)Z^u_{m+1}; \boldsymbol{\mu}_{0:\infty}\big)\Big] \\
        & \leq \mathbb{E} \Big[\Phi^u_{m+1} \big(x_{0:m}, b_m(x_m) + \theta^u_m (x_m, \boldsymbol{\mu}_m)Z^u_{m+1}; \boldsymbol{\mu}_{0:\infty}\big)\Big] \\
        & \quad \text{(by Assumption \ref{A:assump:3} (iv), Lemma \ref{A:lemma:1} (iii), and Lemma \ref{A:lemma:4} (i))} \\
        & \leq \mathbb{E} \Big[\Psi^u_{m+1} \big(x_{0:m}, b_m(x_m) + \theta^u_m (x_m, \boldsymbol{\mu}_m)Z^u_{m+1}; \boldsymbol{\mu}_{0:\infty}\big)\Big] \\
        & \quad \text{(since $\Phi^u_{m+1} \leq \Psi^u_{m+1}$ by assumption)} \\
        & = \Psi^u_m\big(x_{0:m}; \boldsymbol{\mu}_{0:\infty}\big),
    \end{align*}
    which completes the proof of $\Phi^u_m \leq \Psi^u_m$, and one proceeds by a backward  induction. Consequently, for any $u \in I$, we have the following
    \begin{align*}
        \mathbb{E}\Big[F\big(\bar{X}^u_{t_0:t_{M-1}}, \bar{X}^u_{T}, \bar{X}^u_{T}, \cdots\big)\Big] & = \mathbb{E} \Big[\Phi^u_0 \big(\bar{X}^u_{0}; \bar{\boldsymbol{\mu}}_{0:\infty}\big)\Big] \\
        & \leq \mathbb{E} \Big[\Phi^u_0 \big(\bar{Y}^u_{0}; \bar{\boldsymbol{\mu}}_{0:\infty}\big)\Big] \quad \text{(by Lemma \ref{A:lemma:4}, since $X^u_0 \preceq_{cv} Y^u_0$)} \\
        & \leq \mathbb{E} \Big[\Phi^u_0 \big(\bar{Y}^u_{0}; \bar{\boldsymbol{\nu}}_{0:\infty}\big)\Big] \hspace{0.43cm}  \text{(by Lemma \ref{A:lemma:4} and Proposition \ref{A:prop:2})} \\
        & \leq \mathbb{E} \Big[\Psi^u_0 \big(\bar{Y}^u_{0}; \bar{\boldsymbol{\nu}}_{0:\infty}\big)\Big] \hspace{0.41cm} \text{(by \eqref{A:eqn:10})} \\
        & = \mathbb{E}\Big[F\big(\bar{Y}^u_{t_0:t_{M-1}}, \bar{Y}^u_{T}, \bar{Y}^u_{T}, \cdots\big)\Big].
    \end{align*}
In conclusion, the multi-marginal comparison follows by combining the backward comparison of the transition operators with the forward marginal convex order established in Proposition \ref{A:prop:2}.
\end{proof}

\subsection{Functional convex order}
\subsubsection{On the dependence of trajectory}
Recall the notation $t_m := m \cdot h$. Firstly, we define two interpolators as follows, which extends Definition 5.1 in \cite{LiuPages2023} from $T$ to $\infty$.

\begin{definition}
    \textup{(i)} We define the piecewise affine interpolator $i_h: (\mathbb{R}^d)^\infty \rightarrow \mathcal{C}^d_\infty$, i.e. 
    $$
    i_h(x_{0:\infty}) (t) = \frac{1}{h} \big[(t_{m+1} - t)x_m+(t-t_m)x_{m+1}\big], \quad \forall m \in \mathbb{N}, \; \forall t \in [t_m, t_{m+1}].
    $$
    \textup{(ii)} We define the functional interpolator $I_h$ by
    $$
    \mathcal{C}^d_\infty \ni \alpha := \{\alpha_t\}_{t \geq 0} \longmapsto I_h(\alpha) := i_h \big(\alpha_{t_0 :t_\infty}) \in \mathcal{C}^d_\infty.
    $$
\end{definition}

Before stating the main result of this section, i.e. Theorem \ref{A:thm:1}, we first introduce the following result, which can be seen as a variant of Lemma 5.1 in \cite{LiuPages2023} or Lemma 2.2 in \cite{Pages2014}.

\begin{lemma} 
    Under Assumption \ref{A:assump:1}, $I_h(\bar{X}_{T \wedge \cdot}^{u,h})$ weakly converges to $X^u_{T \wedge \cdot}$ as $h \rightarrow 0$ for the $\|\cdot\|_K$-norm topology.
    \label{A:lemma:5}
\end{lemma}

\begin{proof}
    The stated result follows immediately from
    \begin{align*}
        & \mathbb{E} \Big[\big\|I_h(\bar{X}_{T \wedge \cdot}^{u,h})-X_{T \wedge \cdot}^u\big\|^2_K\Big]  = \mathbb{E} \bigg[\int_0^\infty \mathrm{e}^{-Kt} \Big|I_h(\bar{X}^{u,h}_{T \wedge \cdot})(t)-X^u_{t \wedge T}\Big|^2\mathrm{d}t\bigg] \\
        & \quad \leq 2 \mathbb{E} \bigg[\int_0^\infty \mathrm{e}^{-Kt} \Big|I_h(\bar{X}^{u,h}_{T \wedge \cdot})(t)-\bar{X}^{u,h}_{t \wedge T}\Big|^2\mathrm{d}t\bigg] + 2 \mathbb{E} \bigg[\int_0^\infty \mathrm{e}^{-Kt} \Big|\bar{X}^{u,h}_{t \wedge T} -X^u_{t \wedge T}\Big|^2\mathrm{d}t\bigg] \\
        & \quad \leq 2 \mathbb{E} \left[\int_0^{[T]^h} \mathrm{e}^{-Kt} \left(\big|\bar{X}^{u,h}_{[t]^h}-\bar{X}^{u,h}_t\big|^2 + \big|\bar{X}^{u,h}_{[t]^h+h}-\bar{X}^{u,h}_t\big|^2\right)\mathrm{d}t\right] \\
        & \qquad + 2 \mathbb{E} \left[\int_{[T]^h}^T \mathrm{e}^{-Kt} \left(\big|\bar{X}^{u,h}_{[T]^h}-\bar{X}^{u,h}_t\big|^2 + \big|\bar{X}^{u,h}_{T}-\bar{X}^{u,h}_t\big|^2\right)\mathrm{d}t\right] \\
        & \qquad +  2 \mathbb{E} \left[\int_T^{[T]^h + h} \mathrm{e}^{-Kt} \big|\bar{X}^{u,h}_{[T]^h}-\bar{X}^{u,h}_T\big|^2 \mathrm{d}t\right] + \frac{2 \widehat{\gamma}}{K} h^{1 \wedge 2\rho} \\
        & \quad \leq \frac{2 \widehat{\kappa}h}{K} \Big(2+ 3\mathrm{e}^{-K[T]^h}-4\mathrm{e}^{-KT}-\mathrm{e}^{-K[T]^h-Kh}\Big) + \frac{2\widehat{\gamma}h^{1 \wedge 2\rho}}{K} \ \rightarrow 0, 
    \end{align*}
    as $h \rightarrow 0$, where we used Jensen's inequality, the definition of the interpolator $I_h$ as well as Lemma \ref{A:lemma:9}.
\end{proof}

\begin{theorem}
    Suppose that Assumptions \ref{A:assump:1} and \ref{A:assump:2} hold. For every $u \in I$, let $X^u := \{X^u_t\}_{t \geq 0}$, $Y^u := \{Y^u_t\}_{t \geq 0}$ denote the unique solutions of systems \eqref{A:eqn:1} and \eqref{A:eqn:2}, and for every $t \geq 0$, let $\mu^u_t$, $\nu^u_t$ denote the probability distributions of $X^u_t$ and $Y^u_t$. Then, we have
    
    \textup{(a)} \textup{Marginal convex order:} $\mu^u_t \preceq_{cv} \nu^u_t$, $\forall t \geq 0, \; \forall u \in I$.

    \vspace{1mm}
    
    \textup{(b)} \textup{Functional convex order:} for any convex function $F:  (L^2(0, \infty;\mathbb{R}^d), \|\cdot\|_K) \rightarrow \mathbb{R}$ with quadratic growth, one has 
    $$
    \mathbb{E} \Big[F\big(X^u\big)\Big] \leq \mathbb{E} \Big[F\big(Y^u\big)\Big], \quad \forall u \in I.
    $$
    \label{A:thm:1}
\end{theorem}

\begin{proof}
    \textup{(a)} Let $\varphi \in \mathbb{C}_{cv}(\mathbb{R}^d, \mathbb{R})$ with linear growth. Proposition \ref{A:prop:2} implies that $\bar{\mu}^{u,h}_{[t]^h} \preceq_{cv} \bar{\nu}^{u,h}_{[t]^h}$ for every $t \geq 0$, where we used the notation $[t]^h := \lfloor \frac{t}{h}\rfloor \cdot h$. Consequently, we have
    \begin{equation}
        \mathbb{E} \Big[\varphi \big(\bar{X}^{u,h}_{[t]^h}\big)\Big] \leq \mathbb{E} \Big[\varphi\big(\bar{Y}^{u,h}_{[t]^h}\big)\Big], \quad \forall \varphi \in \mathbb{C}_{cv}(\mathbb{R}^d, \mathbb{R}) \text{ with linear growth}.
        \label{A:eqn:57}
    \end{equation}
    It follows from Lemma \ref{A:lemma:9} that $\bar{X}^{u,h}_{[t]^h}$ weakly converges to $X^u_t$ as $h \rightarrow 0$, and $\big\{\bar{X}^{u,h}_{[t]^h}\big\}_h$ is uniformly integrable by the Vitali convergence theorem. Moreover, notice that
    \begin{equation*}
        \Big|\varphi \big(\bar{X}^{u,h}_{[t]^h}\big)\Big| \leq C\Big(1 + \big|\bar{X}^{u,h}_{[t]^h}\big|\Big),
    \end{equation*}
    which further implies the uniform integrability of $\big\{\varphi \big(\bar{X}^{u,h}_{[t]^h}\big)\big\}$. Thus, we have $\mathbb{E} [\varphi (\bar{X}^{u,h}_{[t]^h})] \rightarrow \mathbb{E}[\varphi(X^u_t)]$, and similarly, $\mathbb{E} [\varphi (\bar{Y}^{u,h}_{[t]^h})] \rightarrow \mathbb{E}[\varphi(Y^u_t)]$. One completes the proof of the first part by letting $h \rightarrow 0$ on both sides of \eqref{A:eqn:57}, and by using Lemma \ref{A:lemma:2}.

    \textup{(b)} Firstly, it follows from Proposition \ref{A:prop:1} (i) that $F(X_{T \wedge \cdot}^u)$ and $F(Y_{T \wedge \cdot}^u)$ are in $L^1(\mathbb{P})$ since $F$ has a quadratic growth. We define a function $F_h$ as follows:
    $$
    (\mathbb{R}^d)^\infty \ni x := (x_0, \cdots, x_\infty) \longmapsto F_h(x) := F\big(i_h(x)\big) \in \mathbb{R}.
    $$
    The function $F_h$ is obviously convex since $i_h$ is a linear interpolator, and has a quadratic growth on $(\mathbb{R}^d)^\infty$ in the sense of \eqref{A:eqn:54}. To see the latter, 
    \begin{align*}
        \Big|  F_h(x) \Big| & = \Big| F\big(i_h(x)\big) \Big| \leq C \bigg(1 + \int_0^\infty \mathrm{e}^{-Kt} \big|i_h(x)(t)\big|^2\mathrm{d}t\bigg) \\
        & \leq  C \bigg[1 + \sum_{m=0}^\infty \frac{1}{h}\int_{t_m}^{t_{m+1}} \mathrm{e}^{-Kt}  \Big((t_{m+1} - t)|x_m|^2+(t-t_m)|x_{m+1}|^2\Big)\mathrm{d}t\bigg] \\
        & \leq C \bigg[1 + \sum_{m=0}^\infty \mathrm{e}^{-Kt_m} \frac{|x_m|^2 + |x_{m+1}|^2}{2} h \bigg] =: C \Big(1 + \|x\|_K^2\Big). 
    \end{align*}
    Furthermore, recall that $M := \frac{T}{h}$ if $\frac{T}{h} \in \mathbb{N}$, otherwise $M := \lfloor \frac{T}{h}\rfloor +1$. Then, 
    $$
    F_h \big(\bar{X}^u_{t_0:t_{M-1}}, \bar{X}^u_{T}, \bar{X}^u_{T}, \cdots\big) = F\big(i_h \big(\bar{X}^u_{t_0:t_{M-1}}, \bar{X}^u_{T}, \bar{X}^u_{T}, \cdots\big)\big) = F\big(I_h \big(\bar{X}_{T \wedge \cdot}^u\big)\big).
    $$
    It then follows from Theorem \ref{A:prop:3} that
    \begin{align}
        \mathbb{E}\Big[F\big(I_h \big(\bar{X}_{T \wedge \cdot}^u\big)\big)\Big] & = \mathbb{E}\Big[F\big(i_h \big(\bar{X}^u_{t_0:t_{M-1}}, \bar{X}^u_{T}, \bar{X}^u_{T}, \cdots\big)\big)\Big]  \nonumber \\
        &= \mathbb{E}\Big[F_h\big(\bar{X}^u_{t_0:t_{M-1}}, \bar{X}^u_{T}, \bar{X}^u_{T}, \cdots\big)\Big] \nonumber \\
        & \leq \mathbb{E} \Big[F_h\big(\bar{Y}^u_{t_0:t_{M-1}}, \bar{Y}^u_{T}, \bar{Y}^u_{T}, \cdots\big)\Big] \nonumber \\
        & = \mathbb{E}\Big[F\big(i_h \big(\bar{Y}^u_{t_0:t_{M-1}}, \bar{Y}^u_{T}, \bar{Y}^u_{T}, \cdots\big)\big)\Big] = \mathbb{E} \Big[F\big(I_h \big(\bar{Y}_{T \wedge \cdot}^u\big)\big)\Big]. \label{A:eqn:11}
    \end{align}
    
    The function $F$ is $\|\cdot\|_K$-continuous as it is convex with $\|\cdot\|_K$-quadratic growth, see \cite[Lemma 2.1.1]{Lucchetti2026}. Thanks to Lemma \ref{A:lemma:5}, we have that $I_h(\bar{X}_{T \wedge \cdot}^u)$ weakly converges to $X^u_{T \wedge \cdot}$ for the $\|\cdot\|_K$-norm topology, and moreover, $F(I_h(\bar{X}_{T \wedge \cdot}^u))$ weakly converges to $F(X_{T \wedge \cdot}^u)$ since $F$ is continuous under the $\|\cdot\|_K$-norm topology. Similarly, $F(I_h(\bar{Y}_{T \wedge \cdot}^u))$ weakly converges to $F(Y_{T \wedge \cdot}^u)$. As $F$ has a quadratic growth, we have for some positive constant $C$,
    $$
    \big|F\big(I_h \big(\bar{X}_{T \wedge \cdot}^u\big)\big)\big| \leq C\Big(1+\big\|I_h\big(\bar{X}_{T \wedge \cdot}^u\big)\big\|^2_K\Big).
    $$
    Thanks to Lemma \ref{A:lemma:5}, together with the Vitali convergence theorem, we have the uniform integrability of  $\{\|I_h(\bar{X}_{T \wedge \cdot}^u)\|^2_K\}_h$. Therefore, $F(I_h (\bar{X}_{T \wedge \cdot}^u))$ is also uniformly integrable. Consequently, $\mathbb{E} [F(I_h(\bar{X}_{T \wedge \cdot}^u))] \rightarrow \mathbb{E} [F(X_{T \wedge \cdot}^u)]$. Similarly, $\mathbb{E} [F(I_h(\bar{Y}_{T \wedge \cdot}^u))] \rightarrow \mathbb{E} [F(Y_{T \wedge \cdot}^u)]$. Thus, by letting $h \rightarrow 0$ on both sides of the inequality \eqref{A:eqn:11}, we have
    \begin{equation}
         \mathbb{E} \Big[F\big(X_{T \wedge \cdot}^u\big)\Big] \leq \mathbb{E} \Big[F\big(Y_{T \wedge \cdot}^u\big)\Big], \quad \forall T \geq 0.
         \label{A:eqn:58}
    \end{equation}
    Notice from Lemma \ref{A:lemma:9} (i) that
    \begin{align*}
        \mathbb{E} \Big[\big\|X^u_{T \wedge \cdot} - X^u\big\|^2_K\Big] = \mathbb{E}\bigg[\int_T^\infty \mathrm{e}^{-Kt} \big|X^u_T - X^u_t\big|^2\mathrm{d}t\bigg] \leq \widehat{C} \frac{\mathrm{e}^{-KT}}{K} \longrightarrow 0.
    \end{align*}
    Thus, we conclude the proof by letting $T \rightarrow \infty$ on both sides of \eqref{A:eqn:58} and using the fact that $\{\|X_{T \wedge \cdot}^u\|^2_K: T \geq 0\}$ is uniformly integrable.
\end{proof}

\subsubsection{Extended functional convex order}
We now extend the functional convex order to functionals depending jointly on the individual trajectory and the flow of marginal distributions. The main result of this section is the following corollary.
\begin{corollary}
    Assume that Assumptions \ref{A:assump:1} and \ref{A:assump:2} are in force. For every $u \in I$, let $X^u := \{X^u_t\}_{t \geq 0}$, $Y^u := \{Y^u_t\}_{t \geq 0}$ denote the unique solutions of systems \eqref{A:eqn:1} and \eqref{A:eqn:2}, and for every $t \in [0, \infty)$, let $\mu^u_t$, $\nu^u_t$ denote the probability distributions of $X^u_t$ and $Y^u_t$. In addition, let $\boldsymbol{\mu}_t := \{\mu^u_t\}_{u \in I}$ and $\boldsymbol{\nu}_t := \{\nu^u_t\}_{u \in I}$. Then, for any function $G$: $L^{2}(0,\infty;\mathbb{R}^d) \times  L^{2,c}(0, \infty;\mathcal{H}(\mathbb{R}^d))\rightarrow \mathbb{R}$ satisfying the following conditions: 
    
    \textup{(i)} $G$ is convex in $\alpha$ with quadratic growth in the sense that
    $$
    \big|G\big(\alpha; \{\boldsymbol{\eta}_t\}_{t \geq 0}\big)\big| \leq C \bigg[1+\int_0^\infty \mathrm{e}^{-Kt} |\alpha_t|^2\mathrm{d}t + \sup_{t \geq 0} \boldsymbol{d}_\mathcal{H}^2 (\boldsymbol{\eta}_t, \boldsymbol{\delta}_0)\bigg],
    $$
    
    \textup{(ii)} G is continuous in $\{\boldsymbol{\eta}_t\}_{t \geq 0}$ with respect to the  metric $\boldsymbol{d}_\mathcal{C}$ defined in \eqref{A:eqn:12}, and is nondecreasing in $\{\boldsymbol{\eta}_t\}_{t \geq 0}$ with respect to the convex order in the sense that  for any $\alpha$ and $\boldsymbol{\eta}, \; \tilde{\boldsymbol{\eta}}$ such that $\eta^u_t \preceq_{cv} \tilde{\eta}^u_t$, $\forall u \in I, \; \forall t \geq 0$,
    $$
    G\big(\alpha; \{\boldsymbol{\eta}_t\}_{t \geq 0}\big) \leq G\big(\alpha; \{\tilde{\boldsymbol{\eta}}_t\}_{t \geq 0}\big),
    $$
    one has
    $$
    \mathbb{E} \Big[G\big(X^u; \{\boldsymbol{\mu}_t\}_{t \geq 0}\big)\Big] \leq \mathbb{E} \Big[G\big(Y^u; \{\boldsymbol{\nu}_t\}_{t \geq 0}\big)\Big], \quad \forall u \in I.
    $$
    \label{A:corollary:1}
\end{corollary}

The proof of Corollary \ref{A:corollary:1} relies on the following proposition. Recall that  $M := \frac{T}{h}$ if $\frac{T}{h} \in \mathbb{N}$, otherwise $M := \lfloor \frac{T}{h}\rfloor +1$, and $t_m := m \cdot h$.

\begin{proposition}
    Let $\{\bar{X}^u_{0:\infty}\}_{u \in I}, \; \{\bar{Y}^u_{0:\infty}\}_{u \in I}, \; \bar{\boldsymbol{\mu}}_{0:\infty}, \; \bar{\boldsymbol{\nu}}_{0:\infty}$ be respectively random variables and probability distribution ensembles defined by \eqref{A:eqn:3} and \eqref{A:eqn:4}. Under Assumptions \ref{A:assump:1} and \ref{A:assump:2}, for any function $\tilde{G}$:
    $(\mathbb{R}^d)^\infty \times \mathcal{H}(\mathbb{R}^d)^{\infty} \rightarrow \mathbb{R}$ satisfying the following conditions: 
    
    \textup{(i)} $\tilde{G}$ is convex in $x_{0:\infty}$ with quadratic growth in the sense that
    $$
    \big|\tilde{G}\big(x_{0:\infty}; \boldsymbol{\eta}_{0:\infty}\big)\big| \leq C \Bigg[1+\sum_{m=0}^\infty \mathrm{e}^{-K t_m} \frac{|x_m|^2 + |x_{m+1}|^2}{2} h + \sup_{m \geq 0} \boldsymbol{d}_\mathcal{H}^2 \big(\boldsymbol{\eta}_m, \boldsymbol{\delta}_0\big)\Bigg],
    $$
    
    \textup{(ii)} $\tilde{G}$ is nondecreasing in $\boldsymbol{\eta}_{0:\infty}$ with respect to the convex order in the sense that for any $x_{0:\infty}$ and $\boldsymbol{\eta}_{0:\infty}, \; \tilde{\boldsymbol{\eta}}_{0:\infty}$ such that $\eta^u_i \preceq_{cv} \tilde{\eta}^u_i$, $\forall i \in \mathbb{N}, \; \forall u \in I$, 
    $$
    \tilde{G} \big(x_{0:\infty}; \boldsymbol{\eta}_{0:\infty}\big) \leq \tilde{G} \big(x_{0:\infty}; \tilde{\boldsymbol{\eta}}_{0:\infty}\big),
    $$
    one has
    $$
    \mathbb{E} \Big[\tilde{G}\big(\bar{X}^u_{t_0: t_{M-1}}, \bar{X}^u_{ T},\bar{X}^u_{ T}, \cdots; \bar{\boldsymbol{\mu}}_{0:\infty}\big)\Big] \leq \mathbb{E} \Big[\tilde{G}\big(\bar{Y}^u_{t_0:t_{M-1}},\bar{Y}^u_{ T}, \bar{Y}^u_{ T}, \cdots; \bar{\boldsymbol{\nu}}_{0:\infty}\big)\Big], \quad \forall u \in I.
    $$
    \label{A:prop:4}
\end{proposition}

\begin{proof}
    Similar to $\{\Phi^u_m\}$ and $\{\Psi^u_m\}$, we define
    \begin{equation*}
    \left\{
    \begin{aligned}
        & \widehat{\Phi}^u_m \big(x_{0:M-1}, \underbrace{x_M, \cdots,x_M}_{\text{\#: $m-M+1$}}; \boldsymbol{\mu}_{0:\infty}\big) := \tilde{G}\big(x_{0:M-1}, \underbrace{x_M, x_M, \cdots}_{\text{\#: $\infty$}}; \boldsymbol{\mu}_{0:\infty}\big), \quad m \geq M \text{ or } m = \infty,\\
        & \widehat{\Phi}^u_m\big(x_{0:m}; \boldsymbol{\mu}_{0:\infty}\big) := \mathbb{E} \Big[\widehat{\Phi}^u_{m+1} \big(x_{0:m}, b_m(x_m) + \sigma^u_m (x_m, \boldsymbol{\mu}_m)Z^u_{m+1}; \boldsymbol{\mu}_{0:\infty}\big)\Big], \quad m =M-1, \\
        & \widehat{\Phi}^u_m \big(x_{0:m}; \boldsymbol{\mu}_{0:\infty}\big) := \Big(Q^u_{m+1} \widehat{\Phi}^u_{m+1} \big(x_{0:m}, \cdot; \boldsymbol{\mu}_{0:\infty}\big)\Big)\big(x_m, \sigma^u_m (x_m, \boldsymbol{\mu}_m)\big), \quad  m < M-1; \\
        & \widehat{\Psi}^u_m \big(x_{0:M-1}, \underbrace{x_M, \cdots, x_M}_{\text{\#: $m-M+1$}}; \boldsymbol{\mu}_{0:\infty}\big) := \tilde{G}\big(x_{0:M-1}, \underbrace{x_M, x_M, \cdots}_{\text{\#: $\infty$}}; \boldsymbol{\mu}_{0:\infty}\big), \quad m \geq M \text{ or } m = \infty, \\
        & \widehat{\Psi}^u_m\big(x_{0:m}; \boldsymbol{\mu}_{0:\infty}\big) := \mathbb{E} \Big[\widehat{\Psi}^u_{m+1} \big(x_{0:m}, b_m(x_m) + \theta^u_m (x_m, \boldsymbol{\mu}_m)Z^u_{m+1}; \boldsymbol{\mu}_{0:\infty}\big)\Big], \quad m =M-1, \\
        & \widehat{\Psi}^u_m \big(x_{0:m}; \boldsymbol{\mu}_{0:\infty}\big) := \Big(Q^u_{m+1} \widehat{\Psi}^u_{m+1} \big(x_{0:m}, \cdot; \boldsymbol{\mu}_{0:\infty}\big)\Big)\big(x_m, \theta^u_m (x_m, \boldsymbol{\mu}_m)\big), \quad  m < M-1.
    \end{aligned}
    \right.
    \end{equation*}
     For the case of $m=M-1$, if $t_M > T$, then one modify $b_m$, $\sigma^u_m$, $\theta^u_m$ and $Z^u_{m+1}$ as in \eqref{A:eqn:67}, \eqref{A:eqn:68} and \eqref{A:eqn:69}. 
    
    The rest of the proof proceeds in a similar manner to that of Theorem \ref{A:prop:3}.
\end{proof}

Next, we extend the definition of $i_h$ to the space $\mathcal{H} (\mathbb{R}^d)^\infty$ as follows: for any $t \in [t_m, t_{m+1}]$,
\begin{align*}
    \mathcal{H} (\mathbb{R}^d)^\infty \ni \boldsymbol{\mu}_{0:\infty} \longmapsto  i_h(\boldsymbol{\mu}_{0:\infty})(t) := \left\{\frac{1}{h} \big[(t_{m+1} - t)\mu^u_m+(t-t_m)\mu^u_{m+1}\big]\right\}_{u \in I} \in \mathcal{H}(\mathbb{R}^d), 
\end{align*}
and similarly, for $I_h$, we define
\begin{equation*}
    \mathcal{C}_\infty (\mathcal{H}(\mathbb{R}^d)) \ni \{\boldsymbol{\mu}_t\}_{t \geq 0} \longmapsto I_h\big(\{\boldsymbol{\mu}_t\}_{t \geq 0}\big) := i_h (\boldsymbol{\mu}_{t_0:t_\infty}) \in \mathcal{C}_\infty (\mathcal{H}(\mathbb{R}^d)).
\end{equation*}
Recall that $\bar{\mu}^u_t := \mathcal{L}(\bar{X}^u_t)$, $\bar{\nu}^u_t := \mathcal{L}(\bar{Y}^u_t)$, for $u \in I$, and $\bar{\boldsymbol{\mu}}_t := \{\bar{\mu}^u_t\}_{u \in I}$, $\bar{\boldsymbol{\nu}}_t := \{\bar{\nu}^u_t\}_{u \in I}$. We define for every $t \geq 0$, $\tilde{\bar{\boldsymbol{\mu}}}_t := \{\tilde{\bar{\mu}}^u_t\}_{u \in I} := I_h\big(\{\bar{\boldsymbol{\mu}}_t\}_{t \geq 0}\big) (t)$, and similarly, $\tilde{\bar{\boldsymbol{\nu}}}_t := \{\tilde{\bar{\nu}}^u_t\}_{u \in I} := I_h\big(\{\bar{\boldsymbol{\nu}}_t\}_{t \geq 0}\big) (t)$.

\begin{proof}[Proof of Corollary \ref{A:corollary:1}]
    We start by proving that
    \begin{align}
        \boldsymbol{d}_\mathcal{C} \big(\{\tilde{\bar{\boldsymbol{\mu}}}_t\}_{t \geq 0}, \{\boldsymbol{\mu}_t\}_{t \geq 0}\big) & = \sup_{t \geq 0} \boldsymbol{d}_\mathcal{H} \big(\tilde{\bar{\boldsymbol{\mu}}}_t, \boldsymbol{\mu}_t\big) = \sup_{t \geq 0} \sup_{u \in I} \mathcal{W}_2\big(\tilde{\bar{\mu}}^u_t, \mu^u_t\big) \nonumber \\
        & \leq \sup_{t \geq 0} \sup_{u \in I} \mathcal{W}_2(\tilde{\bar{\mu}}^u_t, \bar{\mu}^u_t) + \sup_{t \geq 0} \sup_{u \in I} \mathcal{W}_2(\bar{\mu}^u_t, \mu^u_t) \nonumber \\
        & \rightarrow 0 \label{A:eqn:17},
    \end{align}
    where we used the  definitions of $\boldsymbol{d}_\mathcal{C}$, $\boldsymbol{d}_\mathcal{H}$ for the first line, and the triangle inequality of $\mathcal{W}_2$ for the second line. To prove \eqref{A:eqn:17}, on the one hand, for $t \in [t_m, t_{m+1}]$, we have
    \begin{align}
        \sup_{t \geq 0} \sup_{u \in I} \mathcal{W}^2_2(\tilde{\bar{\mu}}^u_t, \bar{\mu}^u_t) & \leq \sup_{t \geq 0} \sup_{u \in I}\mathbb{E} \bigg[\Big|\mathds{1}_{\big\{U_m \leq \frac{t_{m+1}-t}{h}\big\}} \bar{X}^u_m +\mathds{1}_{\big\{U_m > \frac{t_{m+1}-t}{h}\big\}} \bar{X}^u_{m+1} -\bar{X}^u_t\Big|^2\bigg] \nonumber \\
        & \leq \sup_{t \geq 0} \sup_{u \in I}\mathbb{E} \bigg[\Big|\bar{X}^u_m -\bar{X}^u_t\Big|^2\bigg] + \sup_{t \geq 0} \sup_{u \in I}\mathbb{E} \bigg[\Big|\bar{X}^u_{m+1}-\bar{X}^u_t\Big|^2\bigg] \nonumber \\
        & \leq 2 \widehat{\kappa} h \ \rightarrow 0, \label{A:eqn:18}
    \end{align}
    as $h \rightarrow 0$, where $\{U_m\}$ is a sequence of random variables with probability distribution $\mathcal{U}([0,1])$, independent of $\{X^u_0, Y^u_0, B^u\}_{u \in I}$, and we used the fact that  $[(t_{m+1} - t)\bar{\mu}^u_m+(t-t_m)\bar{\mu}^u_{m+1}]/h$ is the distribution of the random variable
    \begin{equation*}
        \mathds{1}_{\big\{U_m \leq \frac{t_{m+1}-t}{h}\big\}} \bar{X}^u_m +\mathds{1}_{\big\{U_m > \frac{t_{m+1}-t}{h}\big\}} \bar{X}^u_{m+1}
    \end{equation*}
    for the first line, and Lemma \ref{A:lemma:9} (i) for the third line. On the other hand, also from Lemma \ref{A:lemma:9}, we have 
    \begin{equation*}
         \sup_{t \geq 0} \sup_{u \in I} \mathcal{W}_2^2 (\bar{\mu}^u_t, \mu^u_t) \leq  \sup_{t \geq 0} \sup_{u \in I} \mathbb{E} \Big[\big|X^u_t -\bar{X}^u_t\big|^2\Big] \longrightarrow 0,
    \end{equation*}
    which, combined with \eqref{A:eqn:18}, completes the proof of \eqref{A:eqn:17}.

We define a function $G_h$ as follows
$$
(\mathbb{R}^d)^\infty \times \mathcal{H}(\mathbb{R}^d)^\infty \ni \big(x_{0:\infty}, \boldsymbol{\eta}_{0:\infty}\big) \longmapsto G_h \big(x_{0:\infty}; \boldsymbol{\eta}_{0:\infty}\big) := G\big(i_h(x_{0:\infty}); i_h(\boldsymbol{\eta}_{0:\infty})\big) \in \mathbb{R}.
$$
Then, it follows from Proposition \ref{A:prop:4} that
\begin{align}
    \mathbb{E} \Big[G \big(I_h(\bar{X}_{T \wedge \cdot}^u); \{\tilde{\bar{\boldsymbol{\mu}}}_t\}_{t \geq 0}\big)\Big] & = \mathbb{E} \Big[G \Big(i_h\big(\bar{X}^u_{t_0:t_{M-1}}, \bar{X}^u_{T}, \bar{X}^u_{T}, \cdots\big); i_h\big(\bar{\boldsymbol{\mu}}_{t_0:t_\infty}\big)\Big)\Big] \nonumber \\
    & = \mathbb{E} \Big[G_h \big(\bar{X}^u_{t_0:t_{M-1}}, \bar{X}^u_{T}, \bar{X}^u_{ T}, \cdots; \bar{\boldsymbol{\mu}}_{t_0:t_\infty}\big)\Big] \nonumber \\
    & \leq \mathbb{E} \Big[G_h \big(\bar{Y}^u_{t_0:t_{M-1}}, \bar{Y}^u_{ T}, \bar{Y}^u_{T}, \cdots; \bar{\boldsymbol{\nu}}_{t_0:t_\infty}\big)\Big] \nonumber \\
    & = \mathbb{E} \Big[G \Big(i_h\big(\bar{Y}^u_{t_0:t_{M-1}}, \bar{Y}^u_{ T}, \bar{Y}^u_{T}, \cdots \big); i_h\big(\bar{\boldsymbol{\nu}}_{t_0:t_\infty}\big)\Big)\Big] \nonumber \\
    & = \mathbb{E} \Big[G \big(I_h(\bar{Y}^u_{T \wedge \cdot}); \{\tilde{\bar{\boldsymbol{\nu}}}_t\}_{t \geq 0}\big)\Big]. \label{A:eqn:19}
\end{align}
Under assumptions imposed on $G$, together with Lemma \ref{A:lemma:5} and the convergence result established in \eqref{A:eqn:17}, we have 
$$
\mathbb{E} \Big[G \big(I_h(\bar{X}_{T \wedge \cdot}^u); \{\tilde{\bar{\boldsymbol{\mu}}}_t\}_{t \geq 0}\big)\Big] \longrightarrow \mathbb{E} \Big[G \big(X^u_{T \wedge \cdot}; \{\boldsymbol{\mu}_t\}_{t \geq 0}\big)\Big], \quad \text{as $h \rightarrow 0$},
$$
(idem for $Y$). We conclude the proof by letting $h \rightarrow 0$, and then $T \rightarrow \infty$ on both sides of \eqref{A:eqn:19}.
\end{proof}

\section{Proofs  for well-posedness and Euler schemes}\label{A:sec:1}

We first present an elementary result that will be essential for the proofs to be given in this section.

\begin{lemma}[Lemma 7.1 in \cite{BayrWu2022}]
    Let $y: [0,\infty) \rightarrow [0, \infty)$ be a non-negative differentiable function. Suppose 
    \begin{equation*}
        y(t) - y(r) \leq -a_1\int_r^t y(s) \mathrm{d}s + a_2 \int_r^t \sqrt{y(s)} \mathrm{d}s + \int_r^t a_3(s) \mathrm{d}s, \quad \forall t>r \geq 0,
    \end{equation*}
    for some $a_1 >0$, $a_2 \in \mathbb{R}$ and non-negative and continuous function $a_3$. Then,
    \begin{equation*}
        y(t) \leq \max \Bigg\{y(0), \Bigg(\frac{a_2}{2 a_1}+\sqrt{\frac{\sup_{0\leq s\leq t} a_3(s)}{a_1}+\frac{a_2^2}{4 a_1^2}}\Bigg)^2\Bigg\}, \quad \forall t\geq 0.
    \end{equation*}
    In particular, $y(t) \leq \max \Big\{y(0), \frac{\sup_{0 \leq s\leq t}a_3(s)}{a_1}\Big\}$ if $a_2=0$.
    \label{A:lemma:6}
\end{lemma}

\subsection{Proof of Lemma \ref{A:lemma:8}}

To show that $\Gamma$ is a contraction with respect to the metric $\boldsymbol{d}_\mathcal{C}$, we start by considering for any $\boldsymbol{\mu}, \boldsymbol{\nu} \in L^{2,c} (0,\infty;\mathcal{H}(\mathbb{R}^d))$, and denote
\begin{equation*}
    l^u (t) := \mathbb{E} \Big[\big|X^{u, \boldsymbol{\mu}}_t-X^{u, \boldsymbol{\nu}}_t\big|^2\Big].
\end{equation*}
Using It\^o's formula, we have
\begin{equation*}
    l^u(t) = \mathbb{E} \bigg[\int_0^t 2\big(X^{u, \boldsymbol{\mu}}_s - X^{u,\boldsymbol{\nu}}_s\big) \cdot \widehat{b}^{u, \boldsymbol{\mu}, \boldsymbol{\nu}}_s + \mathrm{tr} \Big(\widehat{\sigma}^{u, \boldsymbol{\mu}, \boldsymbol{\nu}}_s \big(\widehat{\sigma}^{u, \boldsymbol{\mu}, \boldsymbol{\nu}}_s\big)^\top\Big)\mathrm{d}s\bigg],
\end{equation*}
where $\widehat{b}^{u, \boldsymbol{\mu}, \boldsymbol{\nu}}_s$ and $\widehat{\sigma}^{u, \boldsymbol{\mu}, \boldsymbol{\nu}}_s$ are defined as:
\begin{equation*}
    \widehat{b}^{u, \boldsymbol{\mu}, \boldsymbol{\nu}}_s := b(s, X^{u, \boldsymbol{\mu}}_s) - b(s, X^{u, \boldsymbol{\nu}}_s), \quad \text{and }\widehat{\sigma}^{u, \boldsymbol{\mu}, \boldsymbol{\nu}}_s := \sigma (s, X^{u, \boldsymbol{\mu}}_s, [G \boldsymbol{\mu}_s]^u) - \sigma (s, X^{u, \boldsymbol{\nu}}_s, [G \boldsymbol{\nu}_s]^u),
\end{equation*}
respectively. Therefore, $l^u(t)$ is differentiable, and
\begin{equation*}
    l^u(t) - l^u(r) = \mathbb{E} \bigg[\int_r^t 2\big(X^{u, \boldsymbol{\mu}}_s - X^{u,\boldsymbol{\nu}}_s\big) \cdot \widehat{b}^{u, \boldsymbol{\mu}, \boldsymbol{\nu}}_s + \mathrm{tr} \Big(\widehat{\sigma}^{u, \boldsymbol{\mu}, \boldsymbol{\nu}}_s \big(\widehat{\sigma}^{u, \boldsymbol{\mu}, \boldsymbol{\nu}}_s\big)^\top\Big)\mathrm{d}s\bigg]
\end{equation*}
for all $t > r \geq 0$. By \eqref{A:eqn:20}, we have
\begin{equation*}
    \mathbb{E}\Big[\big(X^{u, \boldsymbol{\mu}}_s - X^{u,\boldsymbol{\nu}}_s\big) \cdot \widehat{b}^{u, \boldsymbol{\mu}, \boldsymbol{\nu}}_s\Big] \leq -c_0 \mathbb{E}\Big[\big|X^{u, \boldsymbol{\mu}}_s - X^{u,\boldsymbol{\nu}}_s\big|^2 \Big].
\end{equation*}
Then, using the fact that for any $A \in \mathbb{M}^{d \times q}(\mathbb{R})$, $\|A\|_\text{F} \leq \sqrt{d \wedge q} \|A\|$, where $\|A\|_\text{F}$ denotes the Frobenius norm of $A$, defined as $\|A\|_\text{F} := \sqrt{\sum_{i,j} |a_{ij}|^2}$, we obtain that
\begin{align}
    \mathbb{E}\Big[\mathrm{tr} \Big(\widehat{\sigma}^{u, \boldsymbol{\mu}, \boldsymbol{\nu}}_s \big(\widehat{\sigma}^{u, \boldsymbol{\mu}, \boldsymbol{\nu}}_s\big)^\top\Big)\Big] & \leq (d \wedge q) \mathbb{E} \Big[\big\|\widehat{\sigma}^{u, \boldsymbol{\mu}, \boldsymbol{\nu}}_s\big\|^2\Big] \label{A:eqn:51} \\
    & \leq 2 (d \wedge q) L^2 \mathbb{E} \bigg[\big|X^{u, \boldsymbol{\mu}}_s - X^{u, \boldsymbol{\nu}}_s\big|^2 + \int_I \mathcal{W}_2^2(\mu^u_s, \nu^u_s)\mathrm{d}u\bigg] \nonumber,
\end{align}
where we used the Lipschitz property as well as \eqref{A:eqn:22} for the second line. Combining these two estimates gives
\begin{align*}
    l^u(t) -l^u(r) \leq -2 \big(c_0 - (d \wedge q)L^2\big) \int_r^t l^u(s) \mathrm{d}s + 2 (d \wedge q) L^2 \sup_{s \geq 0} \sup_{u \in I} \mathcal{W}_2^2(\mu^u_s, \nu^u_s) \cdot (t-r).
\end{align*}
Since $l^u(t)$ is non-negative and differentiable with $l^u(0) = 0$, using Lemma \ref{A:lemma:6} with $a_1 := 2 (c_0 - (d \wedge q)L^2)$, $a_2:=0$ and $a_3:= 2 (d \wedge q) L^2 \cdot \sup_{t \geq 0} \sup_{u \in I} \mathcal{W}_2^2(\mu^u_t, \nu^u_t)$, we have
\begin{equation*}
    l^u(t) \leq \frac{(d \wedge q) L^2}{c_0 - (d \wedge q)L^2} \cdot \sup_{t \geq 0} \sup_{u \in I} \mathcal{W}_2^2(\mu^u_t, \nu^u_t) \leq \frac{1}{7} \sup_{t \geq 0} \sup_{u \in I} \mathcal{W}_2^2(\mu^u_t, \nu^u_t)
\end{equation*}
uniformly in $t \geq 0$ and $u \in I$, and this further gives
\begin{align*}
    \boldsymbol{d}^2_\mathcal{C} \big(\Gamma(\boldsymbol{\mu}), \Gamma(\boldsymbol{\nu})\big) & = \sup_{t \geq 0} \sup_{u \in I} \mathcal{W}_2^2\big(\mathcal{L}(X^{u, \boldsymbol{\mu}}_t), \mathcal{L}(X^{u, \boldsymbol{\nu}}_t)\big) \\
    & \leq \sup_{t \geq 0} \sup_{u \in I} \mathbb{E} \Big[\big|X^{u, \boldsymbol{\mu}}_t-X^{u, \boldsymbol{\nu}}_t\big|^2\Big] \\
    & \leq \frac{1}{7} \sup_{t \geq 0} \sup_{u \in I} \mathcal{W}_2^2(\mu^u_t, \nu^u_t) = \frac{1}{7} \boldsymbol{d}^2_\mathcal{C}(\boldsymbol{\mu}, \boldsymbol{\nu}).
\end{align*}
Thus, $\Gamma$ is a contraction from $L^{2,c} (0,\infty;\mathcal{H}(\mathbb{R}^d))$ to itself.

\subsection{Proof of Lemma \ref{A:lemma:7}}
 Notice that the map $\mathbb{R}_+ \ni t \mapsto \bar{\mu}^{u,h}_t \in \mathcal{P}_2(\mathbb{R}^d)$ is uniformly continuous as for any $t \geq s$,
\begin{align*}
    \mathcal{W}_2^2\big(\bar{\mu}^{u,h}_t, \bar{\mu}^{u,h}_s\big) & \leq \mathbb{E} \Big[\big|\bar{X}^{u,h}_t - \bar{X}^{u,h}_s\big|^2\Big] \\
    & = \mathbb{E}\left[\bigg|\int_s^t b\big([r]^h, \bar{X}^{u,h}_{[r]^h}\big) \mathrm{d}r + \int_s^t \sigma\big([r]^h, \bar{X}^{u,h}_{[r]^h}, [G \bar{\boldsymbol{\mu}}^h_{[r]^h}]^u\big)\mathrm{d}B^u_r\bigg|^2\right] \\
    & \leq C (t-s),
\end{align*}
with $C$ independent of $u$, $h$, $t$, and $s$, where we used Lemma \ref{A:lemma:9} (i) for the last inequality. Then, by \cite[Lemma A.1]{Zhang2023}, it suffices to show that for any $t \geq 0$,
\begin{equation}
    \text{$I \ni u \longmapsto \bar{\mu}^{u,h}_t := \mathcal{L}(\bar{X}^{u,h}_t) \in \mathcal{P}(\mathbb{R}^d)$ is measurable}.
    \label{A:eqn:48}
\end{equation}
Firstly, we will prove by a forward induction that
\begin{equation}
    \text{$I \ni u \longmapsto  \mathcal{L}(\bar{X}^{u,h}_{t_m}, B^u) \in \mathcal{P}(\mathbb{R}^d \times \mathcal{C}^q_\infty)$ is measurable for each $m=0,1,2,\cdots$}.
    \label{A:eqn:46}
\end{equation}
This holds for $m=0$ due to the measurability of $u \mapsto \mathcal{L}(X^u_0)$, the i.i.d. property of $\{B^u\}$ and the independence between $\{X^u_0, B^u\}$. Next, suppose that \eqref{A:eqn:46} holds up to $k-1$ for some $k \in \mathbb{N}^*$. To complete the proof of \eqref{A:eqn:46}, it suffices to show that
\begin{equation}
    I \ni u \longmapsto \mathbb{E} \Big[f\big(\bar{X}^{u,h}_{t_k}\big)g(B^u)\Big] \in \mathbb{R}
    \label{A:eqn:47}
\end{equation}
is measurable for all bounded continuous functions $f,g$. From \eqref{A:eqn:3}, we can write $\bar{X}^{u,h}_{t_k}$ in the form of 
$
\bar{X}^{u,h}_{t_k} = H_k\big(u, \bar{X}^{u,h}_{t_{k-1}}, B^u_{t_k}, B^u_{t_{k-1}}\big),
$ for some $H_k:I \times \mathbb{R}^d \times \mathbb{R}^q \times \mathbb{R}^q \rightarrow \mathbb{R}^d$. Notice that $H_k(u, \cdot, \cdot, \cdot)$ is continuous on $\mathbb{R}^d \times \mathbb{R}^q \times \mathbb{R}^q$ for each $u \in I$, and $H_k(\cdot, x, b_1, b_2)$ is measurable on $I$ for each $(x,b_1, b_2) \in \mathbb{R}^d \times \mathbb{R}^q \times \mathbb{R}^q$. Then, $H_k$ is a Carath\'eodory function, and hence is jointly measurable (\cite[Lemma 4.12]{Bouchard2015}). This finishes the verification of \eqref{A:eqn:47}. 

To prove \eqref{A:eqn:48}, it suffices to show that
\begin{equation*}
    I \ni u \longmapsto \mathbb{E} \Big[f\big(\bar{X}^{u,h}_{t}\big)g(B^u)\Big] \in \mathbb{R}
\end{equation*}
is measurable for all bounded continuous functions $f,g$. Similarly, we can write $\bar{X}^{u,h}_t$ in the form of $H_t(u, \bar{X}^{u,h}_{t_m}, B^u_t, B^u_{t_m})$ for some $m \in \mathbb{N}$ and jointly measurable $H_t$. We then complete the proof of \eqref{A:eqn:48} by using \eqref{A:eqn:46}.

\subsection{Proof of Proposition \ref{A:prop:1}}
The proof of Proposition \ref{A:prop:1} follows immediately from the following Lemma \ref{A:lemma:9} and the introduction of the weighted $L^2$-norm $\|\cdot\|_K$, see \eqref{A:eqn:49}. We remark that the dissipativity condition is crucial below, as it prevents the moment and approximation estimates from growing with the time horizon.

\begin{lemma}
    Assume that Assumption \ref{A:assump:1} is valid. 
    \begin{itemize}
        \item[\textup{(i)}] There exist $h_0, \widehat{C} > 0$ such that
        \begin{equation*}
            \sup_{t \geq 0} \sup_{u \in I} \mathbb{E}\Big[\big|X^u_t\big|^2\Big] \vee \sup_{h \in (0, h_0)} \sup_{t \geq 0} \sup_{u \in I} \mathbb{E}\Big[\big|\bar{X}^{u,h}_t\big|^2\Big] \leq \widehat{C}.
        \end{equation*}
        Moreover, there exists a constant $\widehat{\kappa}$ such that for any $h \in (0,h_0)$,
        \begin{equation*}
            \sup_{t \geq 0} \sup_{u \in I} \mathbb{E}\Big[\big|\bar{X}^{u,h}_t - \bar{X}^{u,h}_{t_m}\big|^2\Big] \leq \widehat{\kappa} (t-t_m) \leq \widehat{\kappa} h,
        \end{equation*}
        where we used the notation $t_m := m \cdot h$ with $m := \lfloor \frac{t}{h}\rfloor$.

        \item[\textup{(ii)}] Recall $h_0 > 0$ defined in \textup{(i)}. Then, there exists $\widehat{\gamma} >0$ such that for any $h \in (0, h_0)$,
        \begin{equation*}
            \sup_{t \geq 0} \sup_{u \in I} \mathbb{E} \Big[\big|X^u_t - \bar{X}^{u,h}_t\big|^2\Big] \leq \widehat{\gamma} h^{1 \wedge 2\rho}.
        \end{equation*}
    \end{itemize}
    \label{A:lemma:9}
\end{lemma}

\begin{proof}
\textup{(i)} Using It\^o's formula, we have
\begin{align*}
    \mathbb{E} \Big[\big|X^u_t\big|^2\Big] - \mathbb{E} \Big[\big|X^u_0\big|^2\Big] = \mathbb{E}\bigg[\int_0^t 2 X^u_s \cdot b(s, X^u_s) + \mathrm{tr}\big(\sigma\sigma^\top (s,X^u_s, [G \boldsymbol{\mu}_s]^u)\big) \mathrm{d}s\bigg].
\end{align*}
Therefore, the functions
$$
\alpha^u(t) := \mathbb{E} \Big[\big|X^u_t\big|^2\Big], \qquad \alpha(t) := \int_I \mathbb{E} \Big[\big|X^u_t\big|^2\Big] \mathrm{d}u = \int_I \alpha^u(t) \mathrm{d}u
$$
are differentiable, and
\begin{equation}
    \mathbb{E} \Big[\big|X^u_t\big|^2\Big] - \mathbb{E} \Big[\big|X^u_r\big|^2\Big] = \mathbb{E}\bigg[\int_r^t 2 X^u_s \cdot b(s, X^u_s) + \mathrm{tr}\big(\sigma\sigma^\top (s,X^u_s, [G \boldsymbol{\mu}_s]^u)\big) \mathrm{d}s\bigg]
    \label{A:eqn:21}
\end{equation}
for all $t > r \geq 0$. From \eqref{A:eqn:20}, we have
$$
X^u_s \cdot b(s, X^u_s) \leq -c_0 \big|X^u_s\big|^2 + |b(s,0)|\big|X^u_s\big|,
$$
 and hence 
\begin{align}
    \mathbb{E}\Big[X^u_s \cdot b(s, X^u_s)\Big] \leq \mathbb{E}\Big[-c_0 \big|X^u_s|^2 + |b(s,0)|\big|X^u_s\big|\Big] \leq -c_0 \alpha^u(s) + |b(s,0)| \sqrt{\alpha^u(s)},
    \label{A:eqn:23}
\end{align}
where the last inequality uses Jensen's inequality. For the rest of integrand in \eqref{A:eqn:21}, we deduce by a similar argument as in \eqref{A:eqn:51} that
\begin{align*}
    & \mathrm{tr}\big(\sigma\sigma^\top (s,X^u_s, [G \boldsymbol{\mu}_s]^u)\big) = \big\|\sigma(s,X^u_s, [G \boldsymbol{\mu}_s]^u)\big\|^2_\text{F} \leq (d \wedge q) \big\|\sigma(s,X^u_s, [G \boldsymbol{\mu}_s]^u)\big\|^2 \\
    & \qquad \leq (d \wedge q) \Big(\big\|\sigma(s,0,[G\boldsymbol{\delta}_0]^u)\big\| + L \big|X^u_s\big| + L \cdot \mathcal{WOP}_2\big([G \boldsymbol{\mu}_s]^u,[G\boldsymbol{\delta}_0]^u\big)\Big)^2 \\
    & \qquad \leq 3 (d \wedge q) \Big(\big\|\sigma(s,0,[G\boldsymbol{\delta}_0]^u)\big\|^2 + L^2 \big|X^u_s\big|^2 + L^2 \int_I \mathcal{W}^2_2 \big(\mu^u_s, \delta_0\big) \mathrm{d}u\Big)
\end{align*}
with $\boldsymbol{\delta}_0 :=  \{\delta^u_0\}_{u \in I}$ and $\delta^u_0 := \delta_0, \; \forall u \in I$, where the second line uses the Lipschitz property of $\sigma$, and the last line uses \eqref{A:eqn:22}. Consequently, we obtain that
$$
\mathbb{E} \Big[\mathrm{tr}\big(\sigma\sigma^\top (s,X^u_s, [G \boldsymbol{\mu}_s]^u)\big)\Big] \leq 3 (d \wedge q) \Big(\big\|\sigma(s,0,[G\boldsymbol{\delta}_0]^u)\big\|^2 + L^2 \alpha^u(s) + L^2 \alpha(s)\Big).
$$
Combining this with \eqref{A:eqn:23} gives
\begin{align}
    \alpha^u(t) - \alpha^u(r) \leq & -2\Big(c_0 - \frac{3}{2} L^2 (d \wedge q)\Big) \int_r^t \alpha^u(s)\mathrm{d}s + 2 \kappa_2 \int_r^t \sqrt{\alpha^u(s)} \mathrm{d}s \nonumber \\
    &+ 3L^2 (d\wedge q) \int_r^t \alpha(s) \mathrm{d}s + 3 (d\wedge q)\kappa_2^2 (t-r),\label{A:eqn:24}
\end{align}
where we used \eqref{A:eqn:59}. Integrating over $u \in I$ gives
\begin{equation*}
    \alpha(t) - \alpha(r) \leq -2\big(c_0 - 3L^2(d \wedge q)\big) \int_r^t \alpha(s)\mathrm{d}s + 2\kappa_2 \int_r^t \sqrt{\alpha(s)} \mathrm{d}s + 3 (d\wedge q)\kappa_2^2 (t-r).
\end{equation*}

Since $\alpha(t)$ is non-negative and differentiable, using Lemma \ref{A:lemma:6} with $a_1 :=2(c_0 - 3L^2(d \wedge q)) > 0$, $a_2:=2\kappa_2$, and $a_3 := 3 \kappa_2^2 (d\wedge q)$, we have $\alpha(t) \leq C_1 < \infty$, for any $t \geq 0$, where
\begin{align*}
    C_1  := \int_I \mathbb{E}\Big[\big|X^u_0\big|^2\Big]\mathrm{d}u \vee \left(\frac{\kappa_2}{2 \big(c_0 - 3L^2(d \wedge q)\big)}+\sqrt{\frac{3 \kappa_2^2 (d\wedge q)}{2 \big(c_0 - 3L^2(d \wedge q)\big)}+\frac{\kappa_2^2}{4 \big(c_0 - 3L^2(d \wedge q)\big)^2}}\right)^2.
\end{align*}
Substituting $\alpha(t) \leq C_1$ into \eqref{A:eqn:24} yields
\begin{align}
    \alpha^u(t) - \alpha^u(r) \leq & -2\Big(c_0 - \frac{3}{2} L^2 (d \wedge q)\Big) \int_r^t \alpha^u(s)\mathrm{d}s + 2 \kappa_2 \int_r^t \sqrt{\alpha^u(s)} \mathrm{d}s \nonumber \\
    &+ 3(d\wedge q) (L^2 C_1 + \kappa_2^2)(t-r).
\end{align}
Since $\alpha^u(t)$ is non-negative and differentiable, using Lemma \ref{A:lemma:6} again, we have $\alpha^u(t) \leq C_2$, uniformly in $t \geq 0$ and $u \in I$, with
\begin{align*}
    C_2 := \sup_{u \in I} \mathbb{E}\Big[|X^u_0|^2\Big] \vee \left(\frac{\kappa_2}{2\Big(c_0 - \frac{3}{2} L^2 (d \wedge q)\Big)}+\sqrt{\frac{3(d\wedge q) (L^2 C_1 + \kappa_2^2)}{2\Big(c_0 - \frac{3}{2} L^2 (d \wedge q)\Big)}+\frac{\kappa_2^2}{4\Big(c_0 - \frac{3}{2} L^2 (d \wedge q)\Big)^2}}\right)^2,
\end{align*}
which completes the proof of
\begin{equation*}
    \sup_{u \in I} \sup_{t \geq 0} \mathbb{E} \Big[\big|X^u_t\big|^2\Big] \leq C_2 < \infty.
\end{equation*}

Next, we claim that for some $h_0 \in (0, \infty)$
\begin{equation}
    \sup_{h \in (0,h_0)} \sup_{m \geq 1} \sup_{u \in I} \mathbb{E}\Big[\big|\bar{X}^{u,h}_{t_m}\big|^2\Big] < \infty.
    \label{A:eqn:28}
\end{equation}
To see this, write
\begin{align*}
    & \big|\bar{X}^{u,h}_{t_{m+1}}\big|^2 - \big|\bar{X}^{u,h}_{t_{m}}\big|^2 \\
    & = 2 \bar{X}^{u,h}_{t_{m}} \cdot \big(\bar{X}^{u,h}_{t_{m+1}} - \bar{X}^{u,h}_{t_{m}}\big) + \big|\bar{X}^{u,h}_{t_{m+1}} - \bar{X}^{u,h}_{t_{m}}\big|^2 \\
    & = 2 \bar{X}^{u,h}_{t_{m}} \cdot b(t_m, \bar{X}^{u,h}_{t_m})h + 2 \bar{X}^{u,h}_{t_{m}} \cdot \sigma(t_m, \bar{X}^{u,h}_{t_m}, [G \bar{\boldsymbol{\mu}}^h_{t_m}]^u) \Delta B^u_{t_m} + \big|b(t_m, \bar{X}^{u,h}_{t_m})\big|^2 h^2 \\
    & \quad + \big|\sigma(t_m, \bar{X}^{u,h}_{t_m}, [G \bar{\boldsymbol{\mu}}^h_{t_m}]^u) \Delta B^u_{t_m}\big|^2 + 2h b(t_m, \bar{X}^{u,h}_{t_m}) \cdot \sigma(t_m, \bar{X}^{u,h}_{t_m}, [G \bar{\boldsymbol{\mu}}^h_{t_m}]^u) \Delta B^u_{t_m},
\end{align*}
where we denote $\Delta B^u_{t_m} := B^u_{t_{m+1}} - B^u_{t_m}$, and let $\beta^{u,h}_m := \mathbb{E}[|\bar{X}^{u,h}_{t_m}|^2]$, and $\beta^{h}_m := \int_I \mathbb{E}[|\bar{X}^{u,h}_{t_m}|^2] \mathrm{d}u := \int_I \beta^{u,h}_m \mathrm{d}u$. Using \eqref{A:eqn:20}, \eqref{A:eqn:22}, \eqref{A:eqn:59} and the Lipschitz properties of $b, \sigma$, we have
\begin{equation*}
    \mathbb{E} \Big[\bar{X}^{u,h}_{t_m} \cdot b(t_m, \bar{X}^{u,h}_{t_m})\Big] \leq \mathbb{E} \Big[\kappa_2 \big|\bar{X}^{u,h}_{t_m}\big| -c_0\big|\bar{X}^{u,h}_{t_m}\big|^2\Big] \leq \kappa_2 \sqrt{\beta^{u,h}_m} -c_0 \beta^{u,h}_m,
\end{equation*}
\begin{equation}
    \mathbb{E} \Big[\big|b(t_m, \bar{X}^{u,h}_{t_m})\big|^2\Big] \leq 2\kappa_2^2 + 2L^2 \mathbb{E}\Big[\big|\bar{X}^{u,h}_{t_m}\big|^2\Big]  = 2\kappa_2^2 + 2L^2 \beta^{u,h}_m,
    \label{A:eqn:60}
\end{equation}
\begin{align}
    \mathbb{E} \Big[\big\|\sigma(t_m, \bar{X}^{u,h}_{t_m}, [G \bar{\boldsymbol{\mu}}^h_{t_m}]^u)\big\|^2\Big] & \leq 3\kappa_2^2 + 3L^2 \mathbb{E}\Big[\big|\bar{X}^{u,h}_{t_m}\big|^2\Big] + 3L^2 \mathcal{WOP}_2^2 \big([G \bar{\boldsymbol{\mu}}^h_{t_m}]^u, [G \boldsymbol{\delta}_0]^u\big) \nonumber \\
    & \leq 3\kappa_2^2 + 3L^2 \mathbb{E}\Big[\big|\bar{X}^{u,h}_{t_m}\big|^2\Big] + 3L^2 \int_I \mathcal{W}^2_2 (\bar{\mu}^{u,h}_{t_m}, \delta_0) \mathrm{d}u \nonumber \\
    & \leq 3\kappa_2^2 + 3L^2 \beta^{u,h}_m + 3L^2 \int_I \beta^{u,h}_m \mathrm{d}u \nonumber \\
    & = 3\kappa_2^2 + 3L^2 \beta^{u,h}_m + 3L^2 \beta^{h}_m. \label{A:eqn:61}
\end{align}
Combining these three estimates gives
\begin{align}
    \beta^{u,h}_{m+1} -\beta^{u,h}_m \leq \bigg[3q\big(\kappa_2^2 +L^2 \beta^{u,h}_m + L^2 \beta^h_m\big) -2c_0 \beta^{u,h}_m + 2\kappa_2 \sqrt{\beta^{u,h}_m}\bigg]h + 2(\kappa_2^2 + L^2 \beta^{u,h}_m) h^2.
    \label{A:eqn:27}
\end{align}
Integrating over $u \in I$ yields
\begin{align}
    \beta_{m+1}^h - \beta_m^h & \leq \bigg[3q\kappa_2^2 +\big(6qL^2 -2c_0\big) \beta^{h}_m + 2\kappa_2 \sqrt{\beta^{h}_m}\bigg]h + 2(\kappa_2^2 + L^2 \beta^{h}_m) h^2 \nonumber \\
    & \leq \bigg[3q\kappa_2^2 +\big(8qL^2 -c_0\big) \beta^{h}_m + \frac{\kappa_2^2}{C_3}\bigg]h + 2(\kappa_2^2 + L^2 \beta^{h}_m) h^2 \label{A:eqn:25}
\end{align}
since
$$
2 \kappa_2 \sqrt{\beta^h_m} \leq C_3 \beta^h_m + \frac{\kappa_2^2}{C_3}, \quad \text{with } C_3:= c_0 + 2qL^2 > 0.
$$
Rewrite \eqref{A:eqn:25} as
\begin{equation*}
    \beta^h_{m+1} \leq (1 - \kappa_h h)\beta^h_m + C_4 h, 
\end{equation*}
where $\kappa_h := c_0 - 8qL^2 - 2L^2 h$ and $C_4:= 5 q\kappa_2^2 + \kappa_2^2/C_3$. From \eqref{A:eqn:26}, we can choose $h_0$ such that $\inf_{h \in (0, h_0)} \kappa_h > 0$ and $1-h \kappa_h \in (0, 1)$ for all $h \in (0, h_0)$. Then, for all $h \in (0, h_0)$,
\begin{align*}
    \beta^h_{m+1} & \leq (1-\kappa_h h)\beta_m^h + C_4 h \leq (1-\kappa_h h)^2 \beta^h_{m-1} + (1-\kappa_h h) C_4 h + C_4h \leq \cdots \\
    & \leq (1-\kappa_h h)^{m+1} \beta^h_0 + \sum_{j=0}^m (1-\kappa_h h)^j C_4 h \\
    & \leq \int_I \mathbb{E}\Big[\big|X^u_0\big|^2\Big] \mathrm{d}u + \frac{C_4}{\kappa_{h_0}}=: C_5.
\end{align*}
Substituting this back to \eqref{A:eqn:27} gives
\begin{align*}
    \beta^{u,h}_{m+1} - \beta^{u,h}_m & \leq \bigg[3 q \kappa_2^2 + (3qL^2 - 2c_0)\beta^{u,h}_m + 3qL^2 C_5 + 2\kappa_2 \sqrt{\beta^{u,h}_m}\bigg]h + 2(\kappa_2^2 + L^2\beta^{u,h}_m)h^2 \\
    & \leq \bigg[3q \kappa_2^2 + \big(8qL^2 - c_0\big)\beta^{u,h}_m + 3qL^2 C_5 + \frac{\kappa_2^2}{C_6}\bigg]h + 2(\kappa_2^2 + L^2\beta^{u,h}_m)h^2
\end{align*}
since 
$$
2 \kappa_2 \sqrt{\beta^{u,h}_m} \leq C_6 \beta^{u,h}_m + \frac{\kappa_2^2}{C_6}, \quad \text{with } C_6:= c_0 + 5qL^2 > 0.
$$
Following the same derivation as before, we obtain that
\begin{equation*}
    \beta^{u,h}_{m+1} \leq (1-\kappa_h h)\beta^{u,h}_m + C_7 h \leq \sup_{u \in I} \mathbb{E}\Big[\big|X^u_0\big|^2\Big] + \frac{C_7}{\kappa_{h_0}} =: C_8,
\end{equation*}
where $C_7 := 5q\kappa_2^2 + 3qL^2 C_5 + \kappa_2^2/C_6$, and  one therefore completes the proof of \eqref{A:eqn:28}.

Then, using \eqref{A:eqn:28}-\eqref{A:eqn:61} and  the Cauchy-Schwarz inequality, we have for any $t \in [t_{m}, t_{m+1}]$ and $h \in (0, h_0)$,
\begin{align}
    \mathbb{E}\Big[\big|\bar{X}^{u,h}_t - \bar{X}^{u,h}_{t_m}\big|^2\Big] & \leq 2 \mathbb{E} \Big[\big|b(t_m, \bar{X}^{u,h}_{t_m})\big|^2 (t-t_m)^2 + q \big\|\sigma(t_m, \bar{X}^{u,h}_{t_m}, [G \bar{\boldsymbol{\mu}}^h_{t_m}]^u)\big\|^2 (t-t_m) \Big] \nonumber \\
    & \leq C_9 (t-t_m) \leq C_9 h, \label{A:eqn:29}
\end{align}
where 
$
C_9 := 2q(5 \kappa_2^2 + 5L^2 C_8 + 3L^2 C_5). 
$
Combining \eqref{A:eqn:29} with \eqref{A:eqn:28}, we can easily verify that
\begin{equation}
    \sup_{h \in (0, h_0)} \sup_{t \geq 0} \sup_{u \in I} \mathbb{E}\Big[\big|\bar{X}^{u,h}_t\big|^2\Big] =: C_{10} < \infty,
    \label{A:eqn:32}
\end{equation}
which completes the proof of  part (i) by letting $\widehat{C} := C_2 \vee C_{10}$ and $\widehat{\kappa} := C_9$.

(ii) Recall $h_0$ defined in (i). Firstly, notice that the definition of  continuous-time Euler scheme \eqref{A:eqn:30} implies that
\begin{equation}
    \mathrm{d} \bar{X}^{u,h}_t = b([t]^h, \bar{X}^{u,h}_{[t]^h})\mathrm{d} t + \sigma\big([t]^h, \bar{X}^{u,h}_{[t]^h}, [G \bar{\boldsymbol{\mu}}^h_{[t]^h}]^u\big) \mathrm{d}B^u_t, \quad \bar{X}^{u,h}_0 = X^u_0, 
    \label{A:eqn:43}
\end{equation}
where $[t]^h := \lfloor \frac{t}{h} \rfloor \cdot h$. In addition, we define $\widehat{\sigma}^{u,h}_s$ as follows:
\begin{equation}
    \widehat{\sigma}^{u,h}_s := \sigma \big(s, X^u_s, [G \boldsymbol{\mu}_s]^u\big) - \sigma\big([s]^h, \bar{X}^{u,h}_{[s]^h}, [G \bar{\boldsymbol{\mu}}^h_{[s]^h}]^u\big).
    \label{A:eqn:44}
\end{equation}
Then, using It\^o's formula, we have
\begin{align*}
    \mathbb{E}\Big[\big|X^u_t - \bar{X}^{u,h}_t\big|^2\Big] = \mathbb{E} \bigg[\int_0^t 2\big(X^u_s - \bar{X}^{u,h}_s\big) \cdot \Big(b\big(s, X^u_s\big) - b\big([s]^h, \bar{X}^{u,h}_{[s]^h}\big)\Big) + \mathrm{tr} \Big(\widehat{\sigma}^{u,h}_s \big(\widehat{\sigma}^{u,h}_s\big)^\top\Big) \mathrm{d}s \bigg].
\end{align*}
This implies that the functions
$$
\gamma^{u,h}(t) := \mathbb{E}\Big[\big|X^u_t - \bar{X}^{u,h}_t\big|^2\Big], \quad \text{and } \gamma^h(t) := \int_I \mathbb{E}\Big[\big|X^u_t - \bar{X}^{u,h}_t\big|^2\Big] \mathrm{d}u = \int_I \gamma^{u,h}(t) \mathrm{d}u
$$
are differentiable, and
\begin{align*}
    \gamma^{u,h}(t) -  \gamma^{u,h}(r) & = \mathbb{E} \bigg[\int_r^t 2\big(X^u_s - \bar{X}^{u,h}_s\big) \cdot \Big(b\big(s, X^u_s\big) - b\big([s]^h, \bar{X}^{u,h}_{[s]^h}\big)\Big) + \mathrm{tr} \Big(\widehat{\sigma}^{u,h}_s \big(\widehat{\sigma}^{u,h}_s\big)^\top\Big) \mathrm{d}s \bigg] \\
    & \leq \mathbb{E} \bigg[\int_r^t 2\big(X^u_s - \bar{X}^{u,h}_s\big) \cdot \Big(b\big(s, X^u_s\big) - b\big([s]^h, \bar{X}^{u,h}_{[s]^h}\big)\Big) + (d \wedge q) \big\|\widehat{\sigma}^{u,h}_s\big\|^2 \mathrm{d}s \bigg]
\end{align*}
for all $t > r \geq 0$. By adding and subtracting terms, we have
\begin{align*}
    &\mathbb{E} \Big[\big(X^u_s - \bar{X}^{u,h}_s\big) \cdot \Big(b\big(s, X^u_s\big) - b\big([s]^h, \bar{X}^{u,h}_{[s]^h}\big)\Big)\Big]  \\ 
    & \quad = \mathbb{E} \Big[\big(X^u_s - \bar{X}^{u,h}_s\big) \cdot \Big(b\big(s, X^u_s\big) - b\big(s, \bar{X}^{u,h}_{s}\big)\Big)\Big]  + \mathbb{E} \Big[\big(X^u_s - \bar{X}^{u,h}_s\big) \cdot \Big(b\big(s, \bar{X}^{u,h}_{s}\big) - b\big([s]^h, \bar{X}^{u,h}_{s}\big)\Big)\Big] \\
    & \qquad + \mathbb{E} \Big[\big(X^u_s - \bar{X}^{u,h}_s\big) \cdot \Big(b\big([s]^h, \bar{X}^{u,h}_{s}\big) - b\big([s]^h, \bar{X}^{u,h}_{[s]^h}\big)\Big)\Big] \\
    & \quad \leq -c_0 \mathbb{E} \Big[\big|X^u_s - \bar{X}^{u,h}_s\big|^2\Big] + \tilde{L} \mathbb{E} \Big[\big|X^u_s - \bar{X}^{u,h}_s\big| \big(1+ \big|\bar{X}^{u,h}_s\big|\big)\Big]h^\rho + L \mathbb{E} \Big[\big|X^u_s - \bar{X}^{u,h}_s\big|\big|\bar{X}^{u,h}_s-\bar{X}^{u,h}_{[s]^h}\big|\Big]\\
    & \quad \leq -c_0 \gamma^{u,h}(s) + C_{11} \sqrt{\gamma^{u,h}(s)} h^{\rho \wedge \frac{1}{2}}
\end{align*}
with $C_{11} := \tilde{L}(1+\sqrt{C_{10}})+L\sqrt{C_9}$,  where the fourth line uses \eqref{A:eqn:20} and Assumption \ref{A:assump:1} (i), and the last line uses the Cauchy-Schwarz inequality as well as \eqref{A:eqn:32} and \eqref{A:eqn:29}. Similarly, we have
\begin{equation*}
    \mathbb{E} \Big[\big\|\widehat{\sigma}^{u,h}_s\big\|^2\Big] \leq 6 \tilde{L}^2(1+2C_2) h^{2\rho} + 8L^2 \gamma^{u,h}(s) + 8L^2 \gamma^h(s)+ 16C_9 L^2 h.
\end{equation*}
Combining these two estimates, we have
\begin{align}
    \gamma^{u,h}(t) - \gamma^{u,h}(r) & \leq 2 \int_r^t \Big(-c_0 \gamma^{u,h}(s) + C_{11} \sqrt{\gamma^{u,h}(s)} h^{\rho \wedge \frac{1}{2}}\Big) \mathrm{d}s \nonumber \\
    & \quad + (d\wedge q) \int_r^t \Big[6\tilde{L}^2(1+2C_2)h^{2\rho} + 8L^2 \gamma^{u,h}(s) + 8L^2 \gamma^h(s) + 16C_9 L^2 h\Big]\mathrm{d}s \label{A:eqn:33}
\end{align}
for all $t > r \geq 0$. Integrating over $u \in I$ gives
\begin{equation*}
    \gamma^h(t) - \gamma^h(r) \leq -2\big(c_0 - 8L^2 (d\wedge q)\big) \int_r^t \gamma^h(s) \mathrm{d}s + 2C_{11} h^{\rho \wedge \frac{1}{2}} \int_r^t \sqrt{\gamma^{h}(s)} \mathrm{d}s + C_{12} h^{2\rho \wedge 1} (t-r)
\end{equation*}
with $C_{12} := (6\tilde{L}^2(1+2C_2)+16C_9L^2)(d\wedge q)$. By Lemma \ref{A:lemma:6} with $a_1 = 2(c_0 - 8L^2 (d\wedge q)) > 2\kappa_1 >0$, $a_2 =  2C_{11} h^{\rho \wedge \frac{1}{2}}$ and $a_3 = C_{12} h^{2\rho \wedge 1}$, we have for any $t \geq 0$ and $h \in (0, h_0)$,
\begin{equation*}
    \gamma^h(t) \leq \Bigg(\frac{C_{11}}{ 2\big(c_0 - 8L^2 (d\wedge q)\big)}+\sqrt{\frac{C_{12}}{2\big(c_0 - 8L^2 (d\wedge q)\big)}+\frac{C_{11}^2}{4\big(c_0 - 8L^2 (d\wedge q)\big)^2}}\Bigg)^2 h^{2 \rho \wedge 1} =: C_{13} h^{2 \rho \wedge 1}.
\end{equation*}
Substituting this back to \eqref{A:eqn:33}, we obtain that
\begin{equation*}
    \gamma^{u,h}(t) - \gamma^{u,h}(r) \leq -2\big(c_0-4L^2(d\wedge q)\big) \int_r^t \gamma^{u,h}(s) \mathrm{d}s + 2C_{11} h^{\rho \wedge \frac{1}{2}} \int_r^t \sqrt{\gamma^{u,h}(s)} \mathrm{d}s + C_{14} h^{2\rho \wedge 1} (t-r)
\end{equation*}
with $C_{14} :=(6\tilde{L}^2(1+2C_2)+16C_9L^2+8L^2 C_{13})(d\wedge q)$. Using Lemma \ref{A:lemma:6} again, we obtain that $\gamma^{u,h}(t) \leq C_{15} h^{2\rho \wedge 1}$, uniformly in $u \in I$ and $t \geq 0$, for any $h \in (0,h_0)$, where $C_{15}$ is defined as
$$
C_{15} := \Bigg(\frac{C_{11} }{2\big(c_0-4L^2(d\wedge q)\big)}+\sqrt{\frac{C_{14} }{2\big(c_0-4L^2(d\wedge q)\big)}+\frac{C_{11}^2 }{4\big(c_0-4L^2(d\wedge q)\big)^2}}\Bigg)^2.
$$
Thus, we complete the proof by letting $\widehat{\gamma} := C_{15}$.
\end{proof}

\section{Application to graphon mean-field games} \label{sec:application}
In this section, we provide an example to illustrate how the (extended) functional convex order results can be used to compare value functions in linear-quadratic(LQ) graphon mean-field games (GMFGs). Given two graphons $M, G$, we first define the corresponding graphon-weighted means $Z^{u}, Z^{u,\sigma}: \mathcal{H}(\mathbb{R}) \rightarrow \mathbb{R}$ by
\begin{equation*}
    Z^u(\boldsymbol{\mu}) := \int_I \int_\mathbb{R} M(u,v) z \mu^v(\mathrm{d}z)\mathrm{d}v, \quad \text{and } Z^{u, \sigma}(\boldsymbol{\mu}) := \int_I \int_\mathbb{R} G(u,v) \sigma(z) \mu^v(\mathrm{d}z)\mathrm{d}v.
\end{equation*}
Then, let us consider the following one-dimensional controlled dynamics:
\begin{equation*}
    \mathrm{d} X^{u, \alpha}_t = \alpha^u_t \mathrm{d}t + Z^{u,\sigma} (\boldsymbol{\mu}_t)\mathrm{d}B^u_t, \quad X^{u, \alpha}_0 = x^u \in \mathbb{R},
\end{equation*}
where $\alpha := \{\alpha^u\}_{u \in I}$, and $\alpha^u := \{\alpha^u_t\}_{t \in [0,T]}$ represents the control of agent $u$. The function $\sigma: \mathbb{R} \rightarrow \mathbb{R}_+$ is Lipschitz and convex; $\{B^u\}_{u \in I}$ is a sequence of independent one-dimensional Brownian motions; $\boldsymbol{\mu} := \{\boldsymbol{\mu}_t\}_{t \in [0, T]}$ is a $\mathcal{H}(\mathbb{R})$-valued flow of probability measure ensembles; and the map $I \ni u \mapsto x^u \in \mathbb{R}$ is measurable satisfying $\sup_{u \in I} |x^u|^{2} < \infty$. Agent $u$ aims to minimize her cost function $\mathcal{J}^{u, \boldsymbol{\mu}}$ defined as follows:
\begin{align*}
    \mathcal{J}^{u, \boldsymbol{\mu}}(t,x^u, \alpha^u) := \frac{1}{2}\mathbb{E} \bigg[\int_t^T \Big(\big(X^{u,\alpha}_s - Z^u(\boldsymbol{\mu}_s)\big)^2 + \big(\alpha^u_s\big)^2\Big)\mathrm{d}s + \big(X^{u,\alpha}_T - Z^u(\boldsymbol{\mu}_T)\big)^2\bigg].
\end{align*}

To solve this problem, we first fix $\boldsymbol{\mu}$. The value function for agent $u$, defined as $\mathcal{V}^{u, \boldsymbol{\mu}}(t,x^u) := \inf_{\alpha^u} \mathcal{J}^{u, \boldsymbol{\mu}}(t,x^u, \alpha^u)$, solves the following Hamilton-Jacobi-Bellman (HJB) equation:
\begin{equation}
    \left\{
    \begin{aligned}
        & -\partial_t \mathcal{V}^{u, \boldsymbol{\mu}}(t,x^u) - \mathcal{H}^{u, \boldsymbol{\mu}}\big(t,x^u, \partial_x \mathcal{V}^{u, \boldsymbol{\mu}}(t,x^u), \partial_{x x} \mathcal{V}^{u, \boldsymbol{\mu}}(t,x^u)\big) = 0, \quad \text{in $(0,T) \times \mathbb{R}$}, \\
        & \mathcal{V}^{u, \boldsymbol{\mu}}(T,x^u) = \frac{1}{2} \big(x^u - Z^u(\boldsymbol{\mu}_T)\big)^2, \hspace{5.5cm} \text{in $\mathbb{R}$},
    \end{aligned}
    \right.
    \label{A:eqn:34}
\end{equation}
with the Hamiltonian $\mathcal{H}^{u, \boldsymbol{\mu}}$ defined by
\begin{align*}
    \mathcal{H}^{u, \boldsymbol{\mu}} \big(t,x, p, M\big) := \inf_{\alpha^u_t} \bigg\{\frac{1}{2} \big(\alpha^u_t\big)^2 +p \cdot \alpha^u_t \bigg\} + \frac{1}{2} \big(x - Z^u(\boldsymbol{\mu}_t)\big)^2 + \frac{1}{2} \big(Z^{u,\sigma}(\boldsymbol{\mu}_t)\big)^2 \cdot M .
\end{align*}
Without ambiguity, we will write $\mathcal{J}^u, \mathcal{V}^u, \mathcal{H}^u$ instead of $\mathcal{J}^{u, \boldsymbol{\mu}}, \mathcal{V}^{u, \boldsymbol{\mu}}, \mathcal{H}^{u, \boldsymbol{\mu}}$ to simplify notation. In the LQ setting, a natural candidate for the value function is given by 
\begin{equation*}
    \mathcal{V}^u (t,x) := \frac{\eta^u_t}{2} x^2 + h^u_t x + \beta^u_t.
\end{equation*}
Then, the optimal control for agent $u$ takes the form
\begin{equation*}
    \alpha^{u,*}_t = -  \partial_x \mathcal{V}^{u}(t,x^u) = -\eta^u_t x^u - h^u_t,
\end{equation*}
where the first equality comes from the definition of $\mathcal{H}^u$, and we denote $\alpha^* := \{\alpha^{u,*}\}_{u \in I}$. Substituting the above $\mathcal{V}^u$ and $\alpha^{u,*}$ into \eqref{A:eqn:34}, we obtain the following ordinary differential equation (ODE) system:
\begin{equation}
    \left\{
    \begin{aligned}
        & (\eta^u_t)^\prime = (\eta^u_t)^2 - 1, \hspace{6.33cm} \eta^u_T =1, \\
        & (h^u_t)^\prime = \eta^u_t h^u_t + Z^u(\boldsymbol{\mu}_t), \hspace{5.34cm} h^u_T = -Z^u(\boldsymbol{\mu}_T),\\
        & (\beta^u_t)^\prime = -\frac{1}{2} \big(Z^u(\boldsymbol{\mu}_t)\big)^2 + \frac{1}{2} (h^u_t)^2 - \frac{1}{2}\eta^u_t \big(Z^{u,\sigma}(\boldsymbol{\mu}_t)\big)^2, \qquad \beta^u_T = \frac{1}{2} \big(Z^u(\boldsymbol{\mu}_T)\big)^2.
    \end{aligned}
    \right.
    \label{A:eqn:35}
\end{equation}
Notice that the first equation is a Riccati equation, which has a unique closed-form solution independent of $u \in I$, see (2.50) in \cite{Carmona2018}. Moreover, $\eta^u_t \equiv 1$, for any $t \in [0, T]$.

The previous analysis is valid for any plug-in $\boldsymbol{\mu}$. Now we consider a special $\boldsymbol{\mu}^*$ defined as:
$$
\boldsymbol{\mu}^*_t = \{\mu^{u,*}_t\}_{u \in I} := \big\{\mathcal{L} (X^{u, \alpha^{*}}_t)\big\}_{u \in I}, \qquad \forall t \in [0,T],
$$
where $\{X^{u, \alpha^{*}}\}_{u \in I}$ is the unique strong solution, see \cite[Proposition 2.1]{BayrChak2023}, to
\begin{equation}
    \mathrm{d} X^{u, \alpha^{*}}_t = \big(- X^{u,\alpha^*}_t - h^u_t\big) \mathrm{d}t + Z^{u, \sigma} (\boldsymbol{\mu}^*_t) \mathrm{d}B^u_t, \quad X^{u, \alpha^*}_0 =x^u \in \mathbb{R}.
    \label{A:eqn:37}
\end{equation}
We denote $m^u_t := \mathbb{E} [X^{u, \alpha^*}_t]$, and notice that $m^u_t = \int_\mathbb{R} x \mu^{u,*}_t(\mathrm{d}x)$. Then, to solve the corresponding ODE system \eqref{A:eqn:35}, we turn to verify the unique solvability of the system
\begin{equation}
    \left\{
    \begin{aligned}
        & (h^u_t)^\prime = h^u_t +  \int_I M(u,v) m^v_t \mathrm{d}v, \qquad h^u_T = -  \int_I M(u,v) m^v_T \mathrm{d}v, \\
        & (m^u_t)^\prime = - m^u_t - h^u_t, \hspace{2.43cm} m^u_0 = x^u.
    \end{aligned}
    \right.
    \label{A:eqn:36}
\end{equation}

To solve the system \eqref{A:eqn:36}, we use similar techniques as  \cite[Section 5]{Caines2021} by converting the existence analysis into a fixed-point problem. More precisely, we view $m^u_t = m(u,t)$ as a function of $(u,t)$. Below we derive an equation for $m^u_t$ by eliminating $h^u_t$. Denote the function space $D_\Lambda$ consisting of continuous $\mathbb{R}$-valued functions on $I \times [0, T]$, equipped with the norm $\|\check{m}\| := \sup_{u, t}|\check{m}(u,t)|$. Define the operator $\Lambda$ as follows: for $\check{m} \in D_\Lambda$,
\begin{align*}
    \big(\Lambda \check{m}\big) (u,t) := \int_0^t \mathrm{e}^{r-t} \bigg\{ \mathrm{e}^{r-T} \int_I M(u,v) \check{m}(v,T)\mathrm{d}v + \int_r^T \mathrm{e}^{r-l} \bigg(\int_I M(u,v) \check{m}(v,l)\mathrm{d}v\bigg)\mathrm{d}l\bigg\}\mathrm{d}r.
\end{align*}
If we assume in addition that
\begin{equation}
    \text{$\int_I M(u,v) h(v) \mathrm{d}v$ is continuous in $u \in I$, \quad for any bounded, measurable $h: I \rightarrow \mathbb{R}$}, 
    \label{A:eqn:38}
\end{equation}
see (H5) in \cite{Caines2021} for a similar assumption. Then, $\Lambda$ is from $D_\Lambda$ to itself.

The solution of \eqref{A:eqn:36} further reduces to finding a fixed-point to the equation:
\begin{equation*}
    \check{m}(u,t) = x^u \mathrm{e}^{-t} + \big(\Lambda \check{m}\big)(u,t).
\end{equation*}
Denote $c_M := \sup_{u \in I} \int_I M(u,v)\mathrm{d}v \leq 1$. We have the bound for the operator norm:
\begin{align*}
    \|\Lambda\| & \leq c_M \cdot \sup_{t \in [0, T]} \bigg\{\int_0^t \bigg[\mathrm{e}^{-(T+t-2r)} + \int_r^T \mathrm{e}^{-(t+l-2r)}\mathrm{d}l\bigg]\mathrm{d}r\bigg\} \\
    & = c_M \cdot \big(1-\mathrm{e}^{-T}\big) < 1.
\end{align*}
Hence, $\Lambda$ is a contraction and \eqref{A:eqn:36} has a unique solution. With this solution $\{h^u\}_{u \in I}$, the dynamics \eqref{A:eqn:37} is then determined, and so is $\{\mathcal{L}(X^{u, \alpha^*})\}$. Finally, $\{\beta^u\}$ is determined by the third equation in \eqref{A:eqn:35} as it only depends on $\{h^u\}$, $\{m^u\}$ and $\{\mathcal{L}(X^{u, \alpha^*})\}$. We remark that $\boldsymbol{\mu}^*$ is indeed the unique GMFG equilibrium. This follows from the facts that $\alpha^* := \{\alpha^{u,*}\}_{u \in I}$ is the unique optimal control as the the cost function is strictly convex in $\alpha^u$, and the uniquely solvable system \eqref{A:eqn:36} does not change for different GMFG equilibria. 

Moreover, as the solution $\{m^u,h^u\}_{u}$ of the system \eqref{A:eqn:36} is continuous on $I \times [0,T]$ due to the definition of $D_\Lambda$ and \eqref{A:eqn:38}, the map $t \rightarrow h^u_t$ is Lipschitz continuous thanks to the boundedness of the derivative $(h^u_t)^\prime$. Thus, the optimal dynamics \eqref{A:eqn:37} satisfies Assumption \ref{A:assump:1} (i)-(ii) and Assumption \ref{A:assump:2} (i)-(iii) thanks to Example \ref{A:example:1}. Besides, the unique solvability of \eqref{A:eqn:36} implies that $h^u_t$ depends on $t$ and $\boldsymbol{m} := \{m^u_t\}_{u,t}$, i.e. $h^u_t = h^u(t, \boldsymbol{m})$.

Now let us introduce another controlled dynamics:
\begin{equation}
    \mathrm{d}Y^{u, \alpha}_t = \alpha^u_t \mathrm{d}t + \widehat{Z}^{u, \theta}(\boldsymbol{\nu}_t)\mathrm{d}B^u_t, \quad Y^{u,\alpha}_0 = y^u \in \mathbb{R},
    \label{A:eqn:39}
\end{equation}
and its associated cost function
\begin{equation}
    \mathcal{J}^{u, \boldsymbol{\nu}}(t,y^u, \alpha^u) := \frac{1}{2}\mathbb{E} \bigg[\int_t^T  \Big(\big(Y^{u,\alpha}_s - Z^u(\boldsymbol{\nu}_s)\big)^2 + \big(\alpha^u_s\big)^2\Big)\mathrm{d}s + \big(Y^{u,\alpha}_T - Z^u(\boldsymbol{\nu}_T)\big)^2 \bigg]
    \label{A:eqn:40}
\end{equation}
where $\widehat{Z}^{u, \theta}(\boldsymbol{\nu}): \mathcal{H}(\mathbb{R}) \rightarrow \mathbb{R}$ is defined by 
$$
\widehat{Z}^{u, \theta}(\boldsymbol{\nu}) := \int_I \int_\mathbb{R} K(u,v) \theta(z) \nu^v(\mathrm{d}z)\mathrm{d}v;
$$
the function $\theta: \mathbb{R} \rightarrow \mathbb{R}$ is Lipschitz, and the map $I \ni u \mapsto y^u \in \mathbb{R}$ is measurable satisfying $\sup_{u \in I} |y^u|^{2} < \infty$. Notice that no convexity is imposed on $\theta$. By a similar reasoning as before, there exists a unique GMFG equilibrium to \eqref{A:eqn:39}-\eqref{A:eqn:40}, and denote the optimal control and the corresponding flow of measure ensembles by $\widehat{\alpha}^{*} := \{\widehat{\alpha}^{u, *}\}_{u \in I}$ and $\boldsymbol{\nu}^* := \{\nu^{u,*}\}_{u \in I}$, respectively. 

To ensure the usage of the functional convex order results, we additionally suppose that: \\
$\bullet$ \quad $x^u = y^u$, for any $u \in I$; \\
$\bullet$ \quad for any fixed $u \in I$, we have for any $(v_1, v_2, z_1, z_2) \in I^2 \times \mathbb{R}^2$ that
$$
G(u,v_1)G(u,v_2) \sigma(z_1)\sigma(z_2) \leq K(u,v_1)K(u,v_2) \theta(z_1)\theta(z_2).
$$
Under the first point, we have $
\widehat{\alpha}^{u,*}_t = - Y^{u, \widehat{\alpha}^*}_t - h^u(t, \{\mathbb{E}[Y_t^{u, \widehat{\alpha}^*}]\}_{u,t}) = - Y^{u, \widehat{\alpha}^*}_t -h^u(t, \{\mathbb{E}[X_t^{u, \alpha^*}]\}_{u,t}),
$ where $\{h^u_t\}$ also solves \eqref{A:eqn:36} since the equation of $\{m^u_t\}$, i.e. the expectation of the optimal controlled dynamics, does not change if we switch the diffusion coefficient from $Z^{u, \sigma}$ to $\widehat{Z}^{u, \theta}$. Thus, Assumption \ref{A:assump:2} (iv)-(v) are satisfied by Example \ref{A:example:1}. Consequently, Theorem \ref{A:thm:1} (b) or \eqref{A:eqn:58} implies the comparison of the following two value functions $\mathcal{V}^{u, \boldsymbol{\mu}^*}(x^u),  \mathcal{V}^{u, \boldsymbol{\nu}^*}(y^u)$, which are defined as:
\begin{align*}
    & \mathcal{V}^{u, \boldsymbol{\mu}^*}(x^u) := \inf_{\alpha^u} \mathcal{J}^{u, \boldsymbol{\mu}^*}(0,x^u, \alpha^u) = \mathcal{J}^{u, \boldsymbol{\mu}^*}(0,x^u, \alpha^{u,*}), \\
    & \mathcal{V}^{u, \boldsymbol{\nu}^*}(y^u) := \inf_{\alpha^u} \mathcal{J}^{u, \boldsymbol{\nu}^*}(0,y^u, \alpha^u) = \mathcal{J}^{u, \boldsymbol{\nu}^*}(0,y^u, \widehat{\alpha}^{u,*}).
\end{align*}
More precisely, $\mathcal{V}^{u, \boldsymbol{\mu}^*}(x^u) \leq \mathcal{V}^{u, \boldsymbol{\nu}^*}(y^u)$, $\forall u \in I$.

\bibliographystyle{plain}
\bibliography{GraphonConvexorder}

\end{document}